\documentclass{article}
\usepackage{amsmath}
\usepackage{amsthm}
\usepackage{tikz}
\usepackage{amssymb}
\usepackage{mathtools}
\usepackage{mathabx}
\usepackage{algorithm}
\usepackage{algpseudocode}
\usepackage{hyperref}
\usepackage[toc]{appendix}
\usepackage[nameinlink]{cleveref}
\DeclarePairedDelimiter\ceil{\lceil}{\rceil}
\DeclarePairedDelimiter\floor{\lfloor}{\rfloor}
\DeclareMathAlphabet{\mathbbold}{U}{bbold}{m}{n}

\usepackage[numbers,sort&compress]{natbib}

\def\parens#1{\left(#1\right)}

\newtheorem{theorem}{Theorem}
\newtheorem{lemma}{Lemma}
\newtheorem*{lemma*}{Lemma}
\newtheorem{corollary}{Corollary}
\newtheorem{proposition}{Proposition}

\newtheorem{definition}{Definition}
\newtheorem{claim}{Claim}
\newtheorem*{claim*}{Claim}

\theoremstyle{remark}

\newtheorem*{example*}{Example}

\newcommand{\FF}{{\mathbb{F}}}

\newcommand{\ZZ}{{\mathbb{Z}}}
\newcommand{\RR}{{\mathbb{R}}}
\newcommand{\NN}{{\mathbb{N}}}
\newcommand{\CC}{{\mathbb{C}}}
\newcommand{\one}{{\mathbbold{1}}}
\newcommand{\bone}{{\bar{\one}}}

\newcommand{\hmt}{{\hat{\mu}_{\le \sqrt[3] N}}}
\newcommand{\bmu}{{\bar{\mu}_{\le \sqrt N}}}
\newcommand{\hmu}{{\hat{\mu}_{\le \sqrt N}}}
\newcommand{\kbar}{\bar{k}}
\newcommand{\khat}{\hat{k}}
\newcommand{\pmin}{{p_{\mathrm{min}}}}
\newcommand{\pmax}{{p_{\mathrm{max}}}}
\newcommand{\Conv}{\mathop{\scalebox{1.5}{\raisebox{-0.2ex}{$\ast$}}}}%

\title{Computing $\pi(N)$: An elementary approach in $\tilde{O}(\sqrt{N})$ time}
\author{Dean Hirsch \qquad Ido Kessler \qquad Uri Mendlovic
\\  \href{mailto:primecounting@gmail.com}{primecounting@gmail.com}}
\date{December 2022}

\begin{document}

\maketitle

\begin{abstract}
We present an efficient and elementary algorithm for computing the number of primes up to $N$ in $\tilde{O}(\sqrt N)$ time, improving upon the existing combinatorial methods that require $\tilde{O}(N ^ {2/3})$ time. Our method has a similar time complexity to the analytical approach to prime counting, while avoiding complex analysis and the use of arbitrary precision complex numbers. While the most time-efficient version of our algorithm requires $\tilde{O}(\sqrt N)$ space, we present a continuous space-time trade-off, showing, e.g., how to reduce the space complexity to $\tilde{O}(\sqrt[3]{N})$ while slightly increasing the time complexity to $\tilde{O}(N^{8/15})$. We apply our techniques to improve the state-of-the-art complexity of elementary algorithms for computing other number-theoretic functions, such as the the Mertens function (in $\tilde{O}(\sqrt N)$ time compared to the known $\tilde{O}(N^{0.6})$), summing Euler's totient function, counting square-free numbers and summing primes. Implementation code is provided.
\end{abstract}

\section{Introduction to prime counting}
Our goal is to compute $\pi(N)$, the number of primes not larger than $N$. This problem has a long history. Eratosthenes invented his famous sieve to find all such primes in $\tilde{O}(N)$, that is, linear time up to logarithmic factors. We briefly discuss here the two existing methods to improve this complexity. A comprehensive discussion of both approaches can be found in \cite[Chapter~3.7]{prime_book}.

\subsection{The combinatorial method}
Legendre \cite{legendre} was the first to develop a method to count primes without actually finding them. In the 20-th century his method was used to establish the so-called \emph{combinatorial method} for counting primes \cite{lagarias_combinatorial}. This method requires $\tilde{O}(N ^ {2/3})$ time and $\tilde{O}(\sqrt[3] N)$ space. Later improvements by logarithmic factors \cite{deleglise_combinatorial}\cite{silva_combinatorial} resulted in the current state-of-the-art algorithm for counting primes \cite{old_primecount}.

Even though a complete explanation of the combinatorial methods may be found in the above reference, we give here a quick intuition for this method, as it will be helpful for understanding our method.

The core idea of the combinatorial method is to simulate a sieve algorithm without actually maintaining the state of every single number. Instead, it stores the number of positive integers up to thresholds of the form $\frac{N}{k}$ that are coprime to primes "sieved" so far. We call these numbers \emph{counters}. Simulating the sieve for a single prime can be done in $O(1)$ work per counter using other counters. Further optimizations reduce the number of thresholds and updates, obtaining the mentioned complexity.

\subsection{The analytic method}
An alternative approach to counting primes called the \emph{analytic method} uses complex integration of the logarithm of the Riemann zeta function $\zeta(s)$ to obtain the number of primes up to $N$. \cite{lagarias_analytic} was the first successful algorithm. To the best of our knowledge, later improvements did not change the time complexity of $O(N ^ {1/2 + \epsilon})$ with space complexity $O(N ^ {1/4 + \epsilon})$ \cite{galway-analytic}\cite{platt_analytic}\cite{buthe_analytic}.

We note that even though the analytical approach has the best asymptotic time complexity, the combinatorial approach holds the current record for counting primes. The analytical approach should win for larger $N$, and there are hints that the tipping point is not much farther than the current record $N$. Nevertheless, it is hard to conduct a direct comparison between the approaches because the analytical approach was never analyzed for its exact complexity, to the best of our knowledge.

Another drawback of the analytical approach is the use of arbitrary precision numbers to evaluate a complex integral. These types are very inefficient and incur numerical errors that have to be tracked and bounded.

\subsection{Our contribution}
In this work we present a novel approach for counting primes not larger than $N$. Our approach is elementary in the sense that it does not use the Riemann zeta function or complex analysis. As in the combinatorial approach, our algorithm counts primes in intervals that we call segments. However, unlike existing algorithms, we do not maintain the number of primes in each segment manually. Instead, we do that efficiently using FFT-based convolutions. This improves the time complexity to $\tilde{O}(\sqrt N)$, which is better than $\tilde{O}(N ^ {2/3})$ achieved by the combinatorial approach.

Similar to the analytical method, our method has an error-correction phase, where numbers at a distance of up to $\tilde{O}(\sqrt N)$ from $N$ are analyzed and their erroneous contribution to the result is removed. We believe this resemblance is not a coincidence, and that our approach can be seen as an elementary version of the analytical approach. Unlike the analytical approach, our approach does not require working with high precision complex numbers. Instead, we can run our entire computation using integers. More specifically, we can use finite-field FFT (also known as NTT, number theoretic transform) instead of its complex version. This allows for faster calculation and eliminates the need to track and bound numerical errors. We note that our algorithm indeed avoids using real numbers except for one simple step that is robust to numerical errors.

In \Cref{sec:time-improvement} we improve the time complexity of the basic version of our algorithm, reducing it to  $O\parens{\sqrt N \log N (\log \log N)^{3/2}}$, while using $O\parens{\sqrt{N \log N}}$ space (\Cref{thm:pi-logn}). In \Cref{sec:space-improvement} we improve the space complexity used in our algorithms, introducing a continuous trade-off between the space and time complexities (\Cref{thm:pi-space-time-tradeoff}), achieving, for instance, space complexity of $\tilde{O}\parens{N^{1/3}}$ at the expense of increasing the time complexity to $\tilde{O}\parens{N^{8/15}}$. An extreme case of $O\parens{N^{2/9+\epsilon}}$ space is also introduced, with time complexity of $O\parens{N^{5/9+\epsilon}}$, showing that our algorithm asymptotically dominates the combinatorial method in both time and space requirements.

In \Cref{sec:other-functions} we extend our method to other number-theoretical functions. For example, we show how to compute:

\begin{itemize}
    \item \Cref{thm:mertens-logn}: The Mertens function (sum over the Möbius function) in time
    \[ O\parens{\sqrt N \log N \sqrt{\log \log N}}. \]
    \item \Cref{cor:sum-prime-powers}: The sum of primes up to $N$ in time $\tilde{O}\parens{\sqrt N}$.
    \item \Cref{cor:count-residue-classes}: The number of primes $p\equiv r \pmod{m}$ up to $N$ for any $r, m$ with $m=\tilde{O}(\sqrt N)$, in time $\tilde{O}(\sqrt N)$.
    \item \Cref{thm:sqfree-sum}: The number of square-free numbers up to $N$ in time $\tilde{O}\parens{\sqrt[3]{N}}$.
    \item \Cref{thm:euler-sum}: The sum of Euler's totient function up to $N$ in time $\tilde{O}\parens{\sqrt N}$.
\end{itemize} 

For all these problems, our algorithm improves the time complexity of the state-of-the-art elementary algorithms.

An implementation of the presented algorithm is available at \cite{our_primecounting}.

\subsection{Notations and preliminaries}

We denote the number of prime numbers up to (and including) $n$ by $\pi(n)$.

Much of this work is concerned with functions from the set $\NN$ of natural numbers ($0\not\in \NN$) to some ring $R$, usually $\ZZ$.
Convolution of functions over $\NN$ is understood as Dirichlet convolution: given two functions $f,g:\NN \to R$, their Dirichlet convolution $f\Conv g:\NN \to R$ is defined by
  \[
    (f\Conv g)(n) = \sum_{d | n} f(d) \cdot g(n/d)
  \]

In \Cref{sec:segmentation} we transform any function $f$ over $\NN$ into an array $\bar{f}$ of values starting at index $n=0$.
Convolution of such arrays is understood as regular convolution:
  \[
    (\bar{f}\Conv \bar{g})[n] = \sum_{m=0}^n \bar{f}[m] \cdot \bar{g}[n-m]
  \]

We denote by $\one: \NN \to R$ the constant function $\one(n)=1$.

We denote by $\omega(n)$ the number of distinct primes dividing $n$, by $\tau(n)$ the number of divisors of $n$, by $\mu(n)$ the Möbius function, by $\pmax(n)$ the largest prime dividing $n$, and we set $\pmax(1)=1$.

We use the notation $\tilde{O}(f(n))$ to hide factors of $\log n$. That is, $\tilde{O}(f(n))$ is any function that is bounded above by $O(f(n)\log^c n)$ for some constant $c$.

We assume a word-RAM model of computation with $w$-bit words for $w = \Theta(\log N)$.

\newpage
\setcounter{tocdepth}{2}
\tableofcontents

\section{Basic Algorithm}
In this section we present the essentials of our algorithm. This already achieves $\tilde{O}(\sqrt N)$ time and space complexity. In \Cref{sec:time-improvement} we improve the time complexity by logarithmic factors and in \Cref{sec:space-improvement} we introduce a space-time trade-off.

Our algorithm consists of two phases. The first phase approximately counts primes up to $N$ using convolutions. For this purpose we introduce the concept of smooth Möbius function and log-scale segmentation. The second phase computes the error term in this approximation, recovering the exact value of $\pi(N)$.

\subsection{The smooth Möbius function}\label{sec:smooth-mobius}
As in the elementary approach we start with the set of all numbers up to $N$, represented by $\one$, and then remove from it numbers divisible by primes up to $\sqrt{N}$.
The elementary approach iterates over primes (up to $\sqrt{N}$), and for each prime $p$ removes the numbers divisible by $p$. What remains is the set of numbers up to $N$ not divisible by any prime up to $\sqrt N$, that is, the primes between $\sqrt N$ and $N$, as well as the number $1$.

Each iteration of removing multiples of a prime $p$ can be viewed as Dirichlet convolution with the function:
  \[
    \mu_p(n) = \begin{cases}
        1, & n = 1\\
        -1, & n = p\\
        0, & \text{otherwise}
        \end{cases}
  \]

Indeed, $(\mu_p \Conv  \one)(n)=\sum_{d|n} \mu_p(d) \one(n/d)$, which is $\one(n)=1$ if $p\nmid n$ and $\one(n)-\one(n/p)=0$ if $p\mid n$. This logic continues to apply when we iteratively convolve with $\mu_p$ for different primes $p$.

In our approach we would like to remove numbers divisible by any prime up to $\sqrt{N}$ using a single convolution. 
We do this by computing the convolution of all $\mu_p$ for primes up to $\sqrt{N}$.
The resulting function is equal to the Möbius function $\mu(n)$ for numbers that are $\sqrt N$-smooth and $f(n)=0$ otherwise.
We denote this function by $\mu_{\le \sqrt N}$:
  \[
    \mu_{\le \sqrt N}(n) = \begin{cases}
        (-1)^{\omega(n)}, & n\text{ is square-free and } \pmax(n) \le \sqrt N \\
        0, & \text{otherwise}
        \end{cases}
  \]
Here $\omega(n)$ is the number of distinct prime divisors of $n$, and $\pmax(n)$ is the largest prime dividing $n$ (and we set $\pmax(1)=1$).

Convolving $\mu_{\le \sqrt N}$ and $\one$, we get the set of all numbers satisfying $\pmin(n) > \sqrt N$.
Since we are interested only in numbers not larger than $N$, we may discard (or equivalently, set to 0) values of inputs greater than $N$ in the computed functions $\mu_{\le \sqrt N}$ and $\one$.
The resulting function corresponds to the set of prime numbers in the interval $(\sqrt{N}, N]$ together with the number $1$, since all other numbers have a divisor $\le \sqrt N$ and thus have been removed by the convolution with $\mu_{\le \sqrt N}(n)$. This result is summarized in the following lemma:

\begin{lemma}\label{lem:pi-relation}
\begin{equation*}
\sum_{n=1}^N (\one \Conv \mu_{\le \sqrt N})(n) = \pi(N) - \pi(\sqrt N) + 1
\end{equation*}
\end{lemma}

The $+1$ in \Cref{lem:pi-relation} comes from the fact that in addition to primes between $\sqrt N$ and $N$, the left-hand side also counts the number 1.

In our algorithm we compute the left-hand side of \Cref{lem:pi-relation} in $\tilde{O}(\sqrt N)$ time, compute $\pi(\sqrt N)$ directly in $\tilde{O}(\sqrt N)$ time, and combine them to get the value of $\pi(N)$.

\subsection{Segmentation}\label{sec:segmentation}
\subsubsection{Exponential segmentation and convolution}
The key idea of our method is to manipulate sets of numbers using fast convolution.
As in the classic elementary approach, we simulate sieving using Dirichlet convolution that filters numbers divisible by small primes.
As in existing methods we reduce the number of convolution cells by summing up the function over a range of values.
However, while past methods used thresholds of the form $N/k$, we use a geometric progression as thresholds: $2^{k\Delta}$ where $\Delta$ is a real parameter that controls the segmentation precision and $k=0,1,2,...$ is a running index.
The above segmentation is useful because it approximately translates Dirichlet convolution into sequence convolution that can be carried out using FFT. We now explain how this can be done.

Simply put, we map all the integers in the interval $[2^{k\Delta}, 2^{(k+1)\Delta})$ to the $k$-th cell in an array. 
This way any function $f$ on $\NN$ gives rise to an infinite array where the $k$-th cell is the sum of the function over integers in the corresponding interval.
As the logarithmic precision $\Delta$ gets smaller, the resulting array retains more information of the original function.
We denote the array that corresponds to $f$ by $\bar{f}$, called the \emph{segmentation of $f$}. So:
  \[
    \bar{f}[k] = \sum_{2^{k\Delta} \le n < 2^{(k+1)\Delta}} f(n)
  \]

We further denote by $\kbar (n)$ the index in the array to which $f(n)$ is mapped, that is:
  \[
    \kbar (n) = \floor*{\frac{\log_2 n}{\Delta}}
  \]

In our algorithms, larger values of $\Delta$ allow for faster computations of convolutions, by reducing the lengths of the relevant arrays, at the cost of introducing segmentation errors. We introduce an error correction procedure (discussed in \Cref{sec:error-correction}), necessary for obtaining an exact result, that will take more time as $\Delta$ increases. Thus, ultimately $\Delta$ balances the convolutions time with the error correction time. The optimal value for $\Delta$ will be a function of $N$, and a good value to keep in mind for the basic version of the algorithm is $\Delta\approx \frac{1}{\sqrt N}$.

\subsubsection{Example: Counting odd numbers}
The function $\one$ can be seen as the indicator function of $\NN$. In order to construct the indicator function of the odd numbers we can use the Dirichlet convolution with the function:
  \[
    \mu_2(n) = \begin{cases}
        1, & n = 1\\
        -1, & n = 2\\
        0, & \text{otherwise}
        \end{cases}
  \]
Letting $f = \one\Conv \mu_2$, we get $f(n) = 1$ if $n$ is odd and zero otherwise. That is, $f$ is the indicator function of the odd numbers.

To compute this convolution efficiently, we use exponential segmentation where the thresholds are power of two, so the resolution parameter is $\Delta = 1$ and we get $\bone[k] = 2^k$ while
  \[
    \bar{\mu}_2[k] = \begin{cases}
        1, & k = 0\\
        -1, & k = 1\\
        0, & \text{otherwise}
        \end{cases}
  \]
By convolving the array $\bone$ and $\bar{\mu}_2$ we get that $\bar{f}[k] = 2^{k-1}$ everywhere except for $\bar{f}[0] = 1$, which is indeed the number of odd numbers in the corresponding intervals.

\subsubsection{Segmentation errors}
The result of the above example was accurate. This was true only due to choosing $\Delta$ such that all the nonzero values of $\mu_2$ fall exclusively on segment boundaries. In the general case we cannot satisfy this requirement. Values in the array convolution might miss their correct target cell. For example choosing $\Delta = 1$, and trying to convolve two functions $f$ and $g$, the value of $f(3) \cdot g(3)$ should be accumulated into $\kbar (9) = \floor*{\log_2 9} = 3$. In practice $f(3)$ was represented in $\bar{f}[1]$ since $\kbar (3) = \floor*{\log_2 3} = 1$ and similarly for $g(3)$, so $f(3) \cdot g(3)$ is wrongly accumulated into $\khat  = 2 \cdot \floor*{\log_2 3} = 2$.

We define a new function $\khat (n)$ as the sum of the indices that correspond to the prime factorization of $n$.

\begin{definition}\label{def:khat}
For $n=\prod_i p_i^{e_i}$, we define $\khat(n)$ as
  \begin{equation}
    \khat(n) = \sum_i e_i \kbar(p_i) = \sum_i e_i \floor*{\frac{\log_2 p_i}{\Delta}}
  \end{equation}
\end{definition}

For prime numbers we have $\khat (p) = \kbar (p)$ but in general $\khat (n)$ may be unequal to $\kbar (n)$. In the above example choosing $n=9$ and $\Delta = 1$ gave:
  \[
    3 = \kbar (9) \neq \khat (9) = 2
  \]

We now state a few simple properties of $\kbar$ and $\khat$.

\begin{claim}[Properties of $\kbar$, $\khat$]\label{lem:kbar-khat}
For any $n_1,n_2\in\NN$, the following hold:
\begin{equation}\label{eq:kbar-subadditivity}
\kbar(n_1n_2) - 1 \le \kbar(n_1)+\kbar(n_2) \le \kbar(n_1n_2)
\end{equation}
\begin{equation}\label{eq:khat-additivity}
\khat(n_1n_2) = \khat(n_1) + \khat(n_2)
\end{equation}
Furthermore, for any $n\in\NN$, it holds that
\begin{equation}\label{eq:khat-vs-kbar}
\kbar(n) - \log_2 n \le \khat(n) \le \kbar(n)
\end{equation}
\end{claim}
\begin{proof}
\Cref{eq:kbar-subadditivity} follows from the fact that $\floor*{x+y} - 1 \le \floor*{x} + \floor*{y} \le \floor*{x+y}$ for any $x,y\in\RR$.
\Cref{eq:khat-additivity} follows from the definition of $\khat$. \Cref{eq:khat-vs-kbar} then follows by noting that a number $n$ has at most $\log_2 n$ prime factors, including multiplicity.
\end{proof}

\subsection{Computing the smooth Möbius function}
\subsubsection{The approximated smooth Möbius function}
In order to compute the convolution $\bone \Conv \bmu$, we need $\bmu$. Computing it directly from $\mu_{\le \sqrt N}$ requires iterating over the $\Theta(N)$ nonzero values of $\mu_{\le \sqrt N}$. Instead, we approximate it as $\hmu$, defined as 
\[ \hmu = \bar{\mu}_2 \Conv \bar{\mu}_3 \Conv \ldots \Conv \bar{\mu}_p \]
convolving $\bar{\mu}_p$ over all primes $p\le \sqrt N$. The notation $\hmu$ is chosen to distinguish it from $\bmu$, that is reserved for the segmentation of the function $\mu_{\le \sqrt N}$. We call $\hmu$ the approximated smooth Möbius function.

We further note that since we wish to count primes not larger than $N$, the convolution $\bone \Conv \bmu$ is only of interest for segments not larger than $\kbar (N) = \floor*{\frac{\log_2 N}{\Delta}}$. Therefore the computation is restricted to convolutions of arrays with $O\parens{\frac{\log N}{\Delta}}$ cells.

Naively, computing $\hmu$ can be done with $\tilde{O}(\sqrt N)$ convolutions on arrays of size $\frac{\log N}{\Delta}$, totalling $\tilde{O}\parens{\frac{\sqrt N}{\Delta}}$ time. This is too much, because we would like $\Delta$ to be very small (as will become clearer when discussing the error-correcting phase in \Cref{sec:error-correction}). Instead, we introduce here an alternative derivation that allows for faster computation of $\hmu$.

Let $\delta_{k}$ be an array of zeros everywhere except for a single $1$ at the $k$-th index. Using this notation $\bar{\mu}_p = \delta_{0} - \delta_{\kbar (p)}$ where $\kbar (p)$ is the array index that corresponds to $p$ and
\begin{equation}\label{eq:hmu-convolution}
\hmu = \Conv_{p \le \sqrt N} \bar{\mu}_{p} = \Conv_{p \le \sqrt N} \parens{ \delta_{0} - \delta_{\kbar (p)} }
\end{equation}
where $\Conv_{p \le \sqrt N}$ should be interpreted as the convolution of all the arrays that correspond to primes not larger than $\sqrt N$. We next use associativity of convolution and the fact that $\delta_{n_0} \Conv \delta_{n_1} = \delta_{n_0+n_1}$ to rewrite the convolutions as a sum of delta functions over ordered $r$-tuples of primes:
  \[
    \hmu = \sum_r (-1)^r \sum_{p_1 < p_2 < ... < p_r \le \sqrt N} \delta_{\kbar (p_1) + ... + \kbar (p_r)}
  \]

Using this notation we see that indeed:
  \[
    \hmu = \sum_{n} \mu_{\le \sqrt N}(n) \delta_{\khat (n)}
  \]
which approximates the original:
  \[
    \bmu = \sum_{n} \mu_{\le \sqrt N}(n) \delta_{\kbar (n)}.
  \]

We summarize the above in the following form, that will be useful later:

\begin{claim}\label{thm:hmu-val}
\begin{equation}
\hmu[k] = \sum_{n:\,\khat(n)=k} \mu_{\le \sqrt N}(n)
\end{equation}
\end{claim}

Intuitively, $\hmu$ is equal to $\bmu$ in a world where each prime $p$ is replaced by its approximation $2^{\kbar (p) \Delta}$, which may not be an integer.

For example, choosing $\Delta = 1$ we get $\kbar (2) = \kbar (3) = 1$ and  $\kbar (5) = 2$, so the primes $2, 3$ and $5$ corresponds to $2^1, 2^1$ and $2^2$ respectively. Note that this map is not one-to-one, we now have two "primes" with the value of $2$. Further note that in general we result in non-integer values rather than whole numbers.

In this alternative world natural numbers are replaced by multiplying their modified prime factors. The number $42 = 2 \times 3 \times 7$ is replaced by $2^1 \times 2^1 \times 2^2 = 2^4 = 16$. We ended up representing $42$ as $16$, even though $42$ falls between $32$ and $64$. This is an example of segmentation error.

Taking into account the fact that many numbers may be mapped to the same segment, we must map sets of whole numbers to multi-sets of segmented numbers.
The advantage of this alternative world is that we can use fast convolution of arrays in order to convolve sets or functions of numbers.

\subsubsection{Applying Newton identities}\label{sec:newton-identities}
We rewrite $\hmu$ as
\begin{equation}\label{eq:mu-hat-newton}
\hmu = \sum_r (-1)^r C_r
\end{equation}
where $C_r$ is an array such that $C_r[k]$ counts the numbers $n$ whose $\khat (n)$ falls into segment $k$, that are the product of exactly $r$ \emph{different} primes not larger than $\sqrt N$. Formally, we define $C_r[k]$ as
\[
\big|\{n\in \NN: \khat (n)=k\text{ and }\omega(n)=r\text{ and }\pmax(n)\le \sqrt N\text{ and $n$ is square-free}\}\big|
\]

This is the "indicator" function for the set of numbers $n$ for which $\omega(n)=r$ and $\pmax(n)\le \sqrt N$, in our alternate number system where there can be multiplicities. We also denote this by $C_r = \{\{\khat (p_1\cdot\ldots\cdot p_r)\}\}$ where it is implied that we compute the multi-set of values of the form $\khat (p_1\cdot\ldots\cdot p_r)$ where $p_1 < p_2 < ... < p_r$ are prime numbers not larger than $\sqrt N$. In general we denote by $\{\{x_i\}\}$ the array whose $k$-th index is the number of occurrences of $k$ in the implied sequence $x_i$ (hence, this is just a short-hand for a sum of indicators). The relevant sequence is understood from the context.

Our approach is to compute the arrays $C_r$ efficiently and then to construct $\hmu$ from them. Note that \Cref{eq:mu-hat-newton} includes all multiplicities $r$, but we are only interested in numbers not larger than $N$, so values with $r > \log_2 N$ can be discarded. In other words, we sum up for $r$ from $0$ to $r_{\text{max}} = \floor{\log_2 N}$.

We now present a method of computing the $C_r$ arrays using $O\parens{\log^2 N}$ convolutions, improving the naive method of computing $\hmu$ by convolving individual primes which requires $\tilde{O}\parens{\sqrt N}$ convolutions. Each convolution is of arrays of size $\frac{\log N}{\Delta}$ requiring $\frac{\log N}{\Delta} \log\parens{\frac{\log N}{\Delta}}$ time using FFT. Since $1/\Delta$ will end up being $N^{O(1)}$, this is $O\parens{\frac{\log^2 N}{\Delta}}$ time per convolution. Hence $\hmu$ can be computed in $O\parens{\frac{\log^4 N}{\Delta}}$ time.

\begin{lemma}\label{lem:hmu-log4}
$\hmu$, and hence also $\bone\Conv\hmu$, can be computed in time complexity $O\parens{\sqrt N \log \log N + \frac{\log^4 N}{\Delta}}$.
\end{lemma}
\begin{proof}
We define arrays $E_r$ similarly to $C_r$, that constitute only the $r$-th powers of primes. We denote this by $E_r = \{\{\khat (p^r)\}\} = \{\{r \khat (p)\}\}$.

We begin by finding all primes up to $\sqrt N$ in $O(\sqrt N \log \log N)$ time and constructing the array $C_1=E_1$ that sums their indicator function. To compute $C_2$, we then note that $C_1\Conv C_1$ (convolving as arrays but discarding cells that corresponds to numbers larger than $N$) almost computes $C_2$, with the caveats being that squares of primes are also counted, and that products of different primes $p_1\cdot p_2$ are counted exactly twice. Hence, $2C_2 = C_1\Conv C_1 - E_2$. To then compute $C_3$, we start by computing $C_2\Conv C_1$, counting each intended $p_1 \cdot p_2 \cdot p_3$ exactly 3 times, but this also counts numbers of the form $p_1 \cdot p_2^2$. We address this by removing $E_2 \Conv  C_1$ from the result. However, we have then removed numbers of the form $p^3$ too many times, and we counter this by adding $E_3$ to the result. In total, we obtain $3C_3 = C_2\Conv C_1 - C_1\Conv E_2 + E_3$. These resemble the famous Newton identities, applied for convolutions. Recalling that $C_1=E_1$ and $C_0 = \{\{\khat (1)\}\} = \{\{0\}\} = \delta_0 $, we obtain

\begin{equation}\label{eq:newton-like}
r C_r = \sum_{r'=1}^r (-1)^{r'-1} C_{r-r'} \Conv  E_{r'}
\end{equation}

The arrays $E_r$ can be computed efficiently from $E_1$ by copying values from index $i$ in $E_1$ to index $r\cdot i$ in $E_r$ and discarding numbers larger than $N$.
Computation of $C_r$ by \Cref{eq:newton-like} requires $r-1$ convolutions, assuming $C_{<r}$ have been computed. We can therefore compute $\hmu$ by \Cref{eq:mu-hat-newton} using $O\parens{\log^2 N}$ convolutions, as promised.
\end{proof}

This section introduced the essential ideas for computing $\hmu$. A better algorithm is described in \Cref{sec:better-newton}.

\subsection{Error correction}\label{sec:error-correction}
As mentioned, convolving $\bone$ with $\hmu$ can produce erroneous values in each segment due to rounding errors. We now explain how to mitigate this problem.

The problem, put differently, is that $\overline{f\Conv g} \neq \bar{f} \Conv \bar{g}$ (note that on each side we use a different type of convolution -- on the left-hand side we use Dirichlet convolution, and on the right-hand side we use a simple convolution of two arrays).

While this problem potentially affects all cells in the resulting array $\bar{f} \Conv \bar{g}$, we are not actually interested in their raw value, but rather in their sum: in our setting, we compute $\one \Conv \mu_{\le \sqrt N}$ up to $N$ in order to discard numbers divisible by primes $\le \sqrt N$, and the sum of the convolution will be 1 plus the number of primes between $\sqrt N$ and $N$ (\Cref{lem:pi-relation}).

We now describe the error in our approximation. Using \Cref{thm:hmu-val}, the value of segment $k$ in $\bone\Conv \hmu$ is
\begin{align*}
(\bone\Conv \hmu)[k] &= \sum_{k_1+k_2=k} \bone[k_1]\cdot \hmu[k_2] \\
&= \sum_{k_1+k_2=k} \parens{\sum_{d_1:\kbar (d_1)=k_1} \one(d_1)} \cdot \parens{\sum_{d_2:\khat (d_2)=k_2} \mu_{\le \sqrt N}(d_2)} \\
&= \sum_{d_1,d_2: \kbar (d_1)+\khat (d_2)=k} \one(d_1)\cdot \mu_{\le \sqrt N}(d_2)
\end{align*}

Using $\kbar(d_1)+\khat(d_2) \le \kbar(d_1d_2)$ (which readily follows from \Cref{lem:kbar-khat}), we can split the above sum into two parts:

\begin{align*}
\sum_{k=0}^{\kbar (N)} (&\bone\Conv \hmu)[k] \\
&= \sum_{\kbar (d_1)+\khat (d_2)\le \kbar(N)} \one(d_1)\cdot \mu_{\le \sqrt N}(d_2) \\
&= \sum_{\substack{d_1d_2\le N \\ \kbar (d_1)+\khat (d_2)\le \kbar(N)}} \one(d_1)\cdot \mu_{\le \sqrt N}(d_2) + \sum_{\substack{d_1d_2>N \\ \kbar (d_1)+\khat (d_2)\le \kbar(N)}} \one(d_1)\cdot \mu_{\le \sqrt N}(d_2)\\
&= \sum_{d_1d_2\le N} \one(d_1)\cdot \mu_{\le \sqrt N}(d_2) + \sum_{\substack{d_1d_2>N \\ \kbar (d_1)+\khat (d_2)\le \kbar(N)}} \one(d_1)\cdot \mu_{\le \sqrt N}(d_2)\\
&= \sum_{n=1}^N (\one\Conv \mu_{\le \sqrt N})(n) + \sum_{\substack{d_1d_2>N \\ \kbar (d_1)+\khat (d_2)\le \kbar(N)}} \one(d_1)\cdot \mu_{\le \sqrt N}(d_2)
\end{align*}

Summarizing, we have shown that the error term is

\begin{lemma}[Error term]\label{thm:error-term}
\begin{equation}\label{eq:error-term}
\sum_{k=0}^{\kbar (N)} (\bone\Conv \hmu)[k] - \sum_{n=1}^N (\one\Conv \mu_{\le \sqrt N})(n) = \sum_{\substack{d_1d_2>N \\ \kbar (d_1)+\khat (d_2)\le \kbar(N)}} \one(d_1)\cdot \mu_{\le \sqrt N}(d_2)
\end{equation}
\end{lemma}

Using \Cref{lem:kbar-khat} we now find that

\begin{equation*} \kbar (d_1) + \khat (d_2) \ge \kbar (d_1) + \kbar (d_2) -\log_2 (d_2) \ge \kbar (d_1d_2) - 1 - \log_2(d_1d_2)
\end{equation*}

Therefore, for $n = d_1 \cdot d_2$ to contribute to the error term we must have $n > N$ and yet

\[ \floor*{\frac{\log_2(n)}{\Delta}} - 1 - \log_2(n) \le \floor*{\frac{\log_2 N}{\Delta}} \]

This already fails to hold for $\log_2(n) > \frac{2\Delta + \log_2 N}{1 - \Delta}$. The number of possible values for $d_1\cdot d_2$ is thus at most
\[ 2^{\frac{2\Delta + \log_2 N}{1 - \Delta}} - N = N\parens{2^{\frac{\Delta (2 + \log_2 N)}{1 - \Delta}} - 1} = O\parens{\Delta N \log N } \]

where to justify the approximation $2^x-1=O(x)$ we used the fact that $\Delta$ will end up being much smaller than $\frac{1}{\log N}$. We have proved the following:

\begin{lemma}\label{lem:critical-interval}
Assuming $\Delta = o\parens{\frac{1}{\log N}}$, the pairs $(d_1,d_2)$ contributing to the error term satisfy $d_1\cdot d_2\in (N, N + S]$ where
\[ S = N\parens{2^{\frac{\Delta (2 + \log_2 N)}{1 - \Delta}} - 1} = O\parens{\Delta N \log N } \]
We call $(N, N + S]$ \emph{the critical interval}.
\end{lemma}

To compute the error correction, we first factorize all numbers in the critical interval. Then we iterate over the divisors $d_1$ of each $n$ in the critical interval and compute the error terms. We now describe these two steps in more detail.

\subsubsection{Sieve}\label{sec:ec-sieve}
We sieve over $(N, N + S]$ with primes up to $\sqrt {N + S}$ to find the complete factorization of each number in this interval. This can be done using the famous sieve of Eratosthenes: we iterate over primes $p\le \sqrt {N+S}$ and mark all numbers divisible by powers of $p$. This requires finding the first point in the interval divisible by $p$, then marking jumps of $p$. For any number in the interval that was not fully factored in this way, the unfactored part must be a prime, for otherwise it would have been divisible by a prime below $\sqrt{N+S}$.

The time spent on each prime $p$ is $O\parens{1+\frac{S}{p}}$, so the total time required for sieving is the sum of this over $p\le \sqrt {N+S}$. Recalling that $S = O\parens{\Delta N \log N }$ and that $\sum_{p\le N}\frac{1}{p}=\ln\ln N +O(1)$ (Mertens' second theorem) we have shown the following:

\begin{lemma}\label{lem:ec-sieve}
Fully factoring all numbers in the critical interval can be done in $O\parens{\sqrt N + \Delta N \log N \log \log N}$ time.
\end{lemma}

\subsubsection{Divisor iteration}\label{sec:ec-div-iter}
For each number $n$ in the critical interval we may need to correct contributions of the form $\one(d_1)\cdot \mu_{\le \sqrt N}(d_2)$ given $d_1d_2 = n$. We iterate over all $n$'s square-free divisors $d_2$ with $\pmax(d_2)\le \sqrt N$, and remove the contribution of $\mu(d_2)$ (which is easily computable since we know the factorization of $d_2$) if and only if $\kbar (n/d_2)+\khat (d_2) \le \kbar(N)$, and otherwise do nothing.

\begin{lemma}\label{lem:ec-div-iter}
Divisor iteration can be done in $O\parens{\sqrt N + \Delta N \log^2 N}$ time.
\end{lemma}
\begin{proof}
We bound the time spent iterating over the divisors of all numbers in the critical interval by counting the time spent on each value of the divisor $d_2$. Divisors larger than $\sqrt{N+S}$ can be paired with divisors smaller than that bound, doubling their work. Moreover, each value of $d_2$ may divide no more than $\ceil*{S/d_2}$ values in an interval of size $S$. 

Recalling that $S = O\parens{\Delta N \log N }$ the correction time complexity is bounded by
\[ 2\sum_{d\le \sqrt{N+S}}\parens{1+\frac{S}{d}} = O\parens{\sqrt N + \Delta N \log^2 N}\]
\end{proof}

\Cref{lem:ec-div-iter} is improved upon in \Cref{sec:fewer-div}.

\subsection{Complexity analysis}\label{sec:basic-complexity}

Combining the above results, we have the following

\begin{theorem}\label{thm:pi-log3}
$\pi(N)$ can be computed in $O\parens{\sqrt{N} \log^3 N}$ time.
\end{theorem}
\begin{proof}
Combining \Cref{lem:pi-relation}, \Cref{lem:hmu-log4}, \Cref{lem:ec-sieve} and \Cref{lem:ec-div-iter}, the total time of the algorithm is bounded above by
\begin{equation*}
O\parens{\sqrt N \log \log N + \Delta N \log N \log \log N + \Delta N \log^2 N + \frac{\log^4 N}{\Delta}}
\end{equation*}
This is minimized for $\Delta = \Theta \parens{ \frac{\log N}{\sqrt N} }$ obtaining the required total time complexity.
\end{proof}

\subsection{Basic algorithm summary}\label{sec:basic-algorithm}
The high-level algorithm is summarized in the following pseudocode:

\begin{algorithm}[H]
\algnewcommand{\Var}[1]{\text{#1}}
\caption{Basic algorithm -- high level}\label{alg:basic-high-level}
\begin{algorithmic}
\Function{CountPrimes}{$N$}
\State $\Delta \gets \frac{\log_2 N}{\sqrt N}$
\State $\bone \gets $ \Call{Get\_$\bone$}{$N$, $\Delta$}
\State $\mathcal{P}_{\le \sqrt N} \gets $ primes up to $\sqrt N$
\State $\hmu \gets $ \Call{Get\_$\hmu$}{$N$, $\Delta$, $\mathcal{P}_{\le \sqrt N}$}
\State $\Var{muOneConv}\gets \bone \Conv \hmu$ \Comment{Convolution using FFT}
\State $\Var{approximateCount} \gets \sum_{i=0}^{\kbar(N)} \Var{muOneConv}[i]$
\State $\Var{correction} \gets $ \Call{ErrorCorrection}{$N$, $\Delta$}
\State \Return $\Var{approximateCount} - \Var{correction} + |\mathcal{P}_{\le \sqrt N}| - 1$
\EndFunction
\end{algorithmic}
\end{algorithm}

The subroutines are described in \Cref{alg:basic-phases}.

\begin{algorithm}[H]
\algnewcommand{\Var}[1]{\text{#1}}
\caption{Subroutines of the basic algorithm}\label{alg:basic-phases}
\begin{algorithmic}

\Function{Get\_$\bone$}{$N$, $\Delta$}
\State $\bone \gets \left[\ceil*{2^{\Delta (k+1)}} - \ceil*{2^{\Delta k}} \text{ for }0\le k \le \kbar(N)\right]$ \Comment{$\kbar(n) := \floor*{\frac{\log_2 n}{\Delta}}$}
\State \Return $\bone$
\EndFunction

\Statex
\Function{Get\_$\hmu$}{$N$, $\Delta$, $\mathcal{P}_{\le \sqrt N}$}
\State $C_0 \gets [1,0,0,\ldots,0]\text{ of length }\kbar(N) + 1$
\State $C_1 \gets [0,0,\ldots,0]\text{ of length }\kbar(N) + 1$
\For{$p\in \mathcal{P}_{\le \sqrt N}$}
\State $C_1[\kbar(p)] \gets C_1[\kbar(p)] + 1$
\EndFor
\State $E_1 \gets C_1$
\For{$r=2$ up to $\floor*{\log_2 N}$}
    \State $E_r \gets [0,0,\ldots,0]\text{ of length }\kbar(N)+1$
    \For{$k=0$ up to $\floor*{\frac{\kbar(N)}{r}}$}
        \State $E_r[r\cdot k] \gets E_1[k]$
    \EndFor
    \State $C_r \gets \frac{1}{r}\sum_{r'=1}^r (-1)^{r'-1} C_{r-r'} \Conv  E_{r'}$ \Comment{Convolutions using FFT}
\EndFor
\State $\hmu \gets \sum_{r=0}^{\floor*{\log_2 N}} (-1)^r C_r$ \Comment{Element-wise addition and multiplication}
\State \Return $\hmu$
\EndFunction

\Statex
\Function{ErrorCorrection}{$N$, $\Delta$}
\State $S \gets N\parens{2^{\frac{\Delta (2 + \log_2 N)}{1 - \Delta}} - 1}$
\State{Factorize numbers in $(N,N+S]$}
\State $\Var{correction} \gets 0$
\For{$n \in (N,N+S]$}
    \For{\text{square-free }$d\mid n$}
        \If{$\pmax(d)\le \sqrt N$ and $\kbar(n/d) + \khat(d) \le \kbar(N)$}
            \State $\Var{correction} \gets \Var{correction} + (-1)^{\omega(d)}$
        \EndIf
    \EndFor
\EndFor
\State \Return correction
\EndFunction
\end{algorithmic}
\end{algorithm}
\clearpage

Schematically, the algorithm flow is as follows:
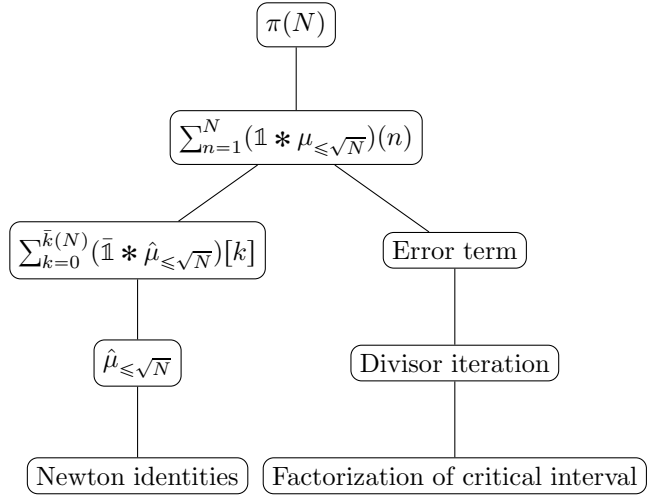
\begin{figure}[H]
\centering
\begin{tikzpicture}[
every node/.style = {shape=rectangle, rounded corners,
                     draw, align=center, fill=white},              
   level 1/.style = {sibling distance = 12em},
   level 2/.style = {sibling distance = 12em},
   level 3/.style = {sibling distance = 20em},
   level 4/.style = {sibling distance = 11em},
]
\node {$\pi(N)$}
    child {node {$\sum_{n=1}^N (\one\Conv \mu_{\le \sqrt N})(n)$}
    child {node {$\sum_{k=0}^{\kbar (N)} (\bone\Conv \hmu)[k]$}
        child {node {$\hmu$}
            child {node{Newton identities}}
        }
    }
    child { node {Error term}
        child { node { Divisor iteration}
            child { node { Factorization of critical interval}
            }
        }
    }};
\end{tikzpicture}
\end{figure}

\section{Time Improvements}\label{sec:time-improvement}

In this section we describe time improvements to the basic prime-counting algorithm. These improvements reduce the time complexity from $O\parens{\sqrt{N} \log^3 N}$ down to $O\parens{\sqrt N \log N (\log \log N)^{3/2}}$.

\subsection{Applying Newton's identities in the Fourier space}\label{sec:better-newton}

We prove a better version of \Cref{lem:hmu-log4}:

\begin{lemma}\label{lem:hmu-log3}
$\hmu$, and hence also $\bone\Conv\hmu$, can be computed in $O\parens{\frac{\log^3 N}{\Delta}}$ time.
\end{lemma}

Note that this lemma improves upon the previous complexity of $O\parens{\frac{\log^4 N}{\Delta}}$.
Before proving \Cref{lem:hmu-log3}, let us first prove a lemma that will be useful.

Recall that we use $E_r$ to denote the array corresponding to the $r$-th power of primes, that is $E_r = \{\{\khat (p^r)\}\} = \{\{r \khat (p)\}\}$. We denote by $\widetilde{E_r}$ the Fourier transform of the array $E_r$.

\begin{lemma}\label{lem:er-e1-fft}
Given $\widetilde{E_1}$ it is possible to produce each $\widetilde{E_r}$ for any $r$ in linear time (that is, $O(1)$ time per entry). This assumes the arrays $E_1$ and $E_r$ are truncated to the same length which is large enough to contain all nonzero entries of the infinite array $E_r$, that is, assuming FFT size is at least $\kbar(N^{r/2})$ elements.
\end{lemma}
\begin{proof}
Note that, when the $E_r$ array is interpreted as a polynomial (by $E_r(x)=\sum_k E_r[k]x^k$), we have $E_r(x)=E_1(x^r)$ by definition. Since FFT essentially evaluates $E_r(x)$ on enough roots of unity, we can read them off from $\widetilde{E_1}$, noting that $x^r$ is also one of the roots of unity evaluated by the FFT of $E_1$.

Note that this fails when the maximal prime power represented in $E_r$ is larger than the array size, which is fixed by FFT size. The maximal prime we use in $\hmu$ is $\le \sqrt N$ so this power is $\le N^{r/2}$ and it is mapped to the cell $\kbar(N^{r/2})$.
\end{proof}

\begin{proof}[Proof of \Cref{lem:hmu-log3}]
Recall \Cref{eq:newton-like}:
\begin{equation*}
r C_r = \sum_{r'=1}^r (-1)^{r'-1} C_{r-r'} \Conv  E_{r'} \tag{\ref{eq:newton-like}}
\end{equation*}

Earlier we assumed each convolution is done individually using FFT. Nevertheless, \Cref{eq:newton-like} may be applied directly in Fourier space:

\begin{equation}\label{eq:newton-like-fourier}
r \widetilde{C_r} = \sum_{r'=1}^r (-1)^{r'-1} \widetilde{C_{r-r'}} \cdot \widetilde{E_{r'}}
\end{equation}

Here the product is a pointwise product of the two Fourier arrays $\widetilde{C_{r-r'}}$ and $\widetilde{E_{r'}}$. Using \Cref{lem:er-e1-fft} we can compute all $\widetilde{E_r}$'s with one application of FFT, then keep all the $\widetilde{C_r}$ and $\widetilde{E_r}$ arrays in Fourier space, use them to compute the Fourier transform of $\hmu$, and then translate back using a single inverse Fourier transform. This reduces the number of FFTs to 2.

However, \Cref{lem:er-e1-fft} has a different problem. Even though we are interested in numbers up to $N$, the arrays we compute correspond to numbers larger than that: the largest prime below $\sqrt{N}$ to the power of the maximal $r$ used in Newton's identities, that is $N^{r_{\text{max}}/2}$. To avoid cyclic overlap, we must keep the arrays long enough to contain the segments corresponding to any product of primes we are using, that is $\kbar(N^{r_{\text{max}}/2}) \approx \frac{r_{\text{max}}}{2\Delta} \log_2 N$, otherwise we end with a meaningless cyclic overlap. Since $r_{\text{max}} = \floor*{\log_2 N}$, the length of the array is $\Theta\parens{\frac{r_{\text{max}}}{\Delta} \log N} = \Theta\parens{\frac{\log^2 N}{\Delta}}$. We call this the problem of \emph{padding}.

Since $1/\Delta$ will be $N^{O(1)}$, each FFT will require $\Theta\parens{\frac{\log^3 N}{\Delta}}$ time to compute.

Nevertheless, the FFTs are no longer the bottleneck: applying \Cref{eq:newton-like-fourier} takes $O(r_{\text{max}}^2)$ time for each of the $\Theta\parens{\frac{\log^2 N}{\Delta}}$ Fourier coefficients, or in total $O\parens{\frac{\log^2 N}{\Delta} \cdot \log^2 N}$, that is $O\parens{\frac{\log^4 N}{\Delta}}$.

To avoid the padding, we note that after applying \Cref{eq:newton-like-fourier} for each $r$ we can apply inverse FFT and discard values larger than $\kbar(N)$. We keep all $\widetilde{C_r}$ and $\widetilde{E_r}$ arrays such that the original $C_r$ and $E_r$ arrays were truncated at $\kbar(N)$, and padded with zeros up to the total length of $2\kbar(N)$. This means we cannot use \Cref{lem:er-e1-fft} at this point, but we will employ it again in a better optimization later (see \Cref{sec:better-better-newton}). Rather, we compute all $\widetilde{E_r}$ with $O\parens{\frac{\log^2 N}{\Delta}}$ time each, totalling at $O\parens{\frac{\log^3 N}{\Delta}}$.

Now for each $r$ we use \Cref{eq:newton-like-fourier} to compute $\widetilde{C_r}$ where the underlying $C_r$ is \emph{not} truncated at $\kbar(N)$. Then, to fix this, we perform an inverse FFT, zero-out all entries after $\kbar(N)$, then perform FFT again to obtain the version of $\widetilde{C_r}$ to be used by future Newton's identities. The number of extra convolutions introduced here is of the same order as done for the $\widetilde{E_r}$'s. Now we still need $O(r)$ operations for each Fourier coefficient in the computation of \Cref{eq:newton-like-fourier}, but there are just $O\parens{\frac{\log N}{\Delta}}$ Fourier coefficients. This brings the substitution time to  $O\parens{\frac{\log^3 N}{\Delta}}$, and this is dominated by the time for the FFTs. Overall we obtain the running time of $O\parens{\frac{\log^3 N}{\Delta}}$.
\end{proof}

In the same way we proved \Cref{thm:pi-log3}, and combining with the above results, we have the following:

\begin{proposition}
$\pi(N)$ can be computed in $O\parens{\sqrt{N}(\log N)^{5/2}}$ time.
\end{proposition}

\subsection{Newton's identities via Newton iteration}\label{sec:better-better-newton}
We follow the steps of the proof of \Cref{lem:hmu-log3}. We note that in \Cref{eq:newton-like-fourier}, each coordinate is independent of the others, the $\widetilde{E_{r'}}$ arrays are known in advance, and we only need to determine the values of $\widetilde{C_r}$. We therefore want to quickly compute the values of a sequence $(c_r)_{r=0}^{r_{\text{max}}}$ defined by 
\begin{equation}\label{eq:newton-like-fourier-entry}
    rc_r=\sum_{r'=1}^r c_{r-r'} e_{r'}
\end{equation}
where $e_{r'}$ are known constants (for simplicity, we merged the $(-1)^{r'-1}$ coefficients with the Fourier coefficients of $E_{r'}$). Furthermore, we assume $c_0=1$ (recall $\widetilde{C_0}$ is the Fourier transform of $C_0=\delta_0$, and hence is an array with a known constant value). We define $e_0=0$, and can then write \Cref{eq:newton-like-fourier-entry} as
\begin{equation}\label{eq:newton-like-fourier-entry2}
    rc_r=\sum_{r'=0}^r c_{r-r'} e_{r'}
\end{equation}

Define the formal series $c(x)=\sum_{r=0}^\infty c_r x^r$ and $e(x)=\sum_{r=0}^\infty e_r x^r$. Then \Cref{eq:newton-like-fourier-entry2} is equivalent to $xc'(x) = e(x)\cdot c(x)$, so $c(x)=c(0) \exp(\int_0^x \frac{e(t)}{t} dt)$. For ease of notation, we let $f(x)=\int_0^x \frac{e(t)}{t}dt$ (note that $\frac{e(t)}{t}$ does not have a $t^{-1}$ coefficient, since $e(0)=0$). We have $c(0)=c_0=1$, so we finally get

\begin{equation}\label{eq:ce-exp}
c(x)=\exp(f(x))
\end{equation}

We are now able to compute the first $r_{\text{max}}$ coefficients of $c(x)$ in time $O\parens{r_{\text{max}} \log r_{\text{max}}}=O(\log N \log\log N)$, since $f(x)$ is computable in linear time, and exponentiation can be done with the asymptotic complexity of $O(1)$ FFTs \cite{fast_exp}.

If we choose to carry all computations in Fourier space without needing to truncate the arrays, there are $O\parens{\frac{\log^2 N}{\Delta}}$ entries in the arrays, bringing the time down to $O\parens{\frac{\log^3 N \log\log N}{\Delta}}$. While this has recovered almost the same complexity of \Cref{lem:hmu-log3}, doing so without truncating the arrays enables the further optimization described in the following section.

\subsection{Partitioning primes to reduce padding}\label{sec:partition-primes}

We incurred a $\log N$ factor in the convolutions complexity due to the extra padding needed for the $\widetilde{C_r}$ arrays. This was needed because we included products of primes much larger than $N$. Here we propose to apply Newton's identities for primes of different sizes separately, allowing for better control over the required padding. Of course, at the end we have to combine all results to obtain the required $\hmu$.

For a prime interval $[p_\text{min}, p_\text{max}]$, let us analyze the complexity of computing the convolutions only on these primes. Here we only need to consider $r$ up to $r_{\text{max}}=\frac{\log_2 N}{\log_2 p_{\text{min}}}$, since multiplying more primes will certainly produce numbers greater than $N$. This requires arrays of size $r_\text{max}\cdot \log_2 p_{\text{max}} /\Delta$ in order to have enough padding. Using the techniques of \Cref{sec:better-better-newton}, each entry's $C_r$ can be computed in $O\parens{r_{\text{max}} \log r_{\text{max}}}$ time. In total the time required for evaluating the Newton identities then is the product of the FFT size with this number.

In addition, we also need to compute $\widetilde{E_1}$ by a single FFT on the $E_1$ arrays, requiring an additional $O\parens{r_\text{max}\cdot \log_2 p_{\text{max}} /\Delta \cdot \log N}$ time  for each such prime interval.

Together, these give a total time of 
\begin{equation}\label{eq:prime-interval-runtime}
O\parens{\frac{r_{\text{max}} \log p_\text{max}}{\Delta} \parens{\log N + r_{\text{max}} \log r_{\text{max}}}}
\end{equation}
for each prime interval.

We choose to partition the primes at $\log_2 p = \frac{\log_2 N}{2^m}$ for $1\le m \le \log_2 \log_2 N$. Using \Cref{eq:prime-interval-runtime} with $\log_2 p_\text{max}=\frac{\log_2 N}{2^m}$ and $\log_2 p_\text{min}=\frac{\log_2 N}{2^{m+1}}$ (and hence $r_\text{max}=2^{m+1})$, implies running time of $O\parens{\frac{\log N}{\Delta} \parens{\log N + m 2^m}}$. Summing over the values of $m$ up to $\log_2 \log_2 N$, we obtain a time bound of $O\parens{\frac{\log^2 N \log \log N}{\Delta}}$.

This is summarized in the following improvement of \Cref{lem:hmu-log3}. 

\begin{lemma}\label{lem:hmu-log2}
$\hmu$, and hence also $\bone\Conv\hmu$, can be computed in time
\[O\parens{\sqrt N \log \log N + \frac{\log^2 N \log \log N}{\Delta}}\]
\end{lemma}

\subsection{Shrinking the critical interval}
We explain why the size $S = O\parens{\Delta N \log N}$ of the critical interval described in \Cref{lem:critical-interval} can be reduced.

\begin{lemma}\label{lem:smaller-critical-interval}
The size $S$ of the critical interval can be taken as 
\[S=O\parens{\Delta N \frac{\log N}{\log \log N}}\]
\end{lemma}
\begin{proof}
Indeed, to get there we used the bound $\kbar(d_1)+\khat(d_2) \ge \kbar(d_1d_2)-1-\log_2(d_1d_2)$, arising from the fact that $d_2$ can have at most $\log_2 d_2$ divisors. However, since we only care about square-free values of $d_2$, the largest number of primes dividing $d_2$ can be at most $O\parens{\frac{\log d_2}{\log\log d_2}}$, since the product of the first $m$ primes behaves like $\exp(\Theta(m \log m))$ (this follows from classical bounds on the first Chebyshev function). Therefore, the inequality is changed to

\begin{equation}\kbar(d_1)+\khat(d_2) \ge \kbar(d_1d_2)-1-O\parens{\frac{\log(d_1d_2)}{\log\log(d_1d_2)}}
\end{equation}

and thus we can take $S=O\parens{\Delta N \frac{\log N}{\log \log N}}$.\end{proof}

\subsection{Useful bounds on number-theoretic functions}\label{sec:bounds}
Before discussing additional improvements, we present here several bounds for later use.

We first cite a weaker version of a theorem of Shiu \cite{sum-small-interval-bound}, from which other results will readily follow.

\begin{lemma}[\cite{sum-small-interval-bound}]\label{lem:sum-small-interval-bound}
Let $f$ be a multiplicative function (that is, $f(ab)=f(a)f(b)$ whenever $\gcd(a,b)=1$).

If there exist $c, d$ such that $f(p^\ell)\le c \cdot \ell^{d}$ for all $\ell\ge 1$ and prime $p$, then for any $\epsilon>0$, if $N^{\epsilon} < S < N$, the following holds

\[
\sum_{n=N}^{N+S}f(n) = O\parens{\frac{S}{\log N} \exp\parens{\sum_{p\le 2N} \frac{f(p)}{p}}}
\]
\end{lemma}
\begin{proof}
This directly follow from \cite[Theorem~1]{sum-small-interval-bound}, where the sufficiency of the condition $f(p^\ell)\le c \cdot \ell^{d}$ is also explained there, in the remark immediately following the theorem statement.
\end{proof}

\begin{lemma}\label{lem:bound-divisors}
Let $\tau(n)$ be the number of divisors of $n \in \NN$. For any $\epsilon>0$, if $N^{\epsilon} < S < N$, then
\[ \sum_{n=N+1}^{N+S} \tau(n) = O(S \log N) \]
\end{lemma}

\begin{proof}
We apply \Cref{lem:sum-small-interval-bound} to the function $f=\tau$. This is easily seen to be a multiplicative function. Also, $\tau(p^\ell) = \ell+1$ satisfies the condition in the theorem. Hence:
\begin{align*}
\sum_{n=N+1}^{N+S} \tau(n) &= O\parens{\frac{S}{\log N} \exp\parens{\sum_{p\le 2N} \frac{2}{p}}} \\ &= O\parens{\frac{S}{\log N} \exp\parens{2 \ln \ln (2N) + O(1)}}
\\ &= O\parens{S \log N}
\end{align*}
\end{proof}

\begin{lemma}\label{lem:bound-omega}
For any $\epsilon>0$, if $N^{\epsilon} < S < N$, then
\[ \sum_{n=N+1}^{N+S} \omega(n) = O(S \log \log N) \]
\end{lemma}
\begin{proof}
Each $n\in(N, N+S]$ can be divisible by at most $1/\epsilon + 1$ primes larger than $S > N^{\epsilon}$, so these primes contribute at most $O(S)$ to the sum. Primes $\le S$ contribute at most

\[ 
  \sum_{p\le S} \ceil*{\frac{S}{p}} < \sum_{p \le S} \parens{1+\frac{S}{p}} = 
  O(S + S\log \log S) = O(S \log \log N)
\]
\end{proof}

\begin{lemma}\label{lem:bound-num-div-on-high-omega}
For any $\epsilon>0$, if $N^{\epsilon} < S < N$, then for any $k$, let $\mathcal{S}_k$ be the set of integers in $(N,N+S]$ having at least $k$ distinct prime factors. Then summing up the number of square-free divisors $2^{\omega(n)}$ for all $n \in \mathcal{S}_k$ satisfies
\[ \sum_{n \in \mathcal{S}_k} 2^{\omega(n)} = \sum_{\substack{n \in (N,N+S]\\ \omega(n)\ge k}} 2^{\omega(n)} = O\parens{\frac{S \log^3 N}{2^k}} \]
\end{lemma}
\begin{proof}
We argue that:
\[
 \sum_{n \in \mathcal{S}_k} 2^{\omega(n)} \le
 \sum_{n=N+1}^{N+S} 2 ^ {2\omega(n) - k}
\]
Since if $\omega(n) < k$, then $n$ contributes 0 to the left hand side, and otherwise we have $2 ^ {\omega(n)} \le 2 ^ {2\omega(n) - k}$.

We apply \Cref{lem:sum-small-interval-bound} to the function $4 ^ {\omega(n)}$ on the interval, giving
\begin{align*}
  \sum_n 4 ^ {\omega(n)} 
  &= O\parens{\frac{S}{\log N} \exp\parens{\sum_{p\le 2N} \frac{4}{p}}} \\ 
  &= O\parens{\frac{S}{\log N} \exp\parens{4 \ln \ln (2N) + O(1)}}
\\ &= O\parens{S \log ^ 3 N}
\end{align*}
We divide by $2^k$ to obtain the desired bound.
\end{proof}

\subsection{Testing fewer divisors}\label{sec:fewer-div}
We now revisit the divisor iteration and improve upon the bound given in \Cref{lem:ec-div-iter}. The idea is that for a given $n$ in the critical interval $(N, N+S]$, we only need to iterate over square-free divisors $d$ with $\pmax(d)\le \sqrt N$ for which $\kbar(n/d)+\khat(d) \le \kbar(N)$. Since $\kbar(n/d)+\khat(d)\ge \kbar(n)-1-\omega(d)$, for each $n$ we only need to iterate over divisors $d$ with $\omega(d)\ge \kbar(n) - \kbar(N) - 1$.

\begin{definition}[critical divisor]
For $n$ in the critical interval $(N, N+S]$, a divisor $d$ of $n$ is called a \emph{critical divisor} if it is square-free and
\[\omega(d)\ge\kbar(n) - \kbar(N) - 1\]
\end{definition}

Rephrasing the previous logic, we have shown that only critical divisors of $n$ can contribute to the error correction.

This section has two main goals. The first is to show that we can efficiently iterate over only the critical divisors of each $n$ in the critical interval. The second is to tightly bound the number of critical divisors to produce better running-time guarantees.

Denote by $D(n)$ the number of critical divisors of $n$.

\begin{lemma}\label{lem:critical-div-iter}
For each $n$, we can iterate all critical divisors of $n$ in $O(\omega(n)+D(n))$ time.
\end{lemma}
\begin{proof}
In the sieving phase we find the prime factors of each $n\in (N,N+S]$. Then, given $n$, we can iterate over square-free divisors $d$ of $n$, starting with the square-free part of $n$, then recursively finding primes to remove from $d$ as long as the number of remaining primes is at least $\kbar(n) - \kbar(N) - 1$, essentially iterating a binary tree of square-free divisors in a depth-first search.
\end{proof}

Combining \Cref{lem:smaller-critical-interval} and \Cref{lem:bound-omega} bounds the first part of the total work:

\begin{equation}\label{eq:omega-sum-critical-interval}
\sum_{n=N+1}^{N+S} \omega(n) = O(\Delta N \log N)
\end{equation}

It remains to tightly bound the number of critical divisors $D(n)$, which is then a bound on the total work done when iterating the critical divisors in the manner described in \Cref{lem:critical-div-iter}.

\begin{lemma}\label{lem:total-critical-divisors}
\[ \sum_{n = N+1}^{N+S} D(n) = O(\Delta N \log N \log \log N) \]
\end{lemma}
\begin{proof}
Recall that the critical interval $(N, N+S]$ consists of segments of size $O(\Delta N)$, where in each segment the value of $\kbar(n)$ is constant. On the $k$-th segment we have $\kbar(n) - \kbar(N) - 1 = k$, so in this segment the critical divisors are square-free $d$'s that satisfy:
\[ \omega(d) \ge \kbar(n) - \kbar(N) - 1 = k \]

For $k < 3 \log_2 \log_2 N$, we pessimistically assume all the divisors are critical. The amount of numbers in those segments is $O(\Delta N \log \log N)$. The total number of divisors of those numbers is bounded by $O(\Delta N \log N \log \log N)$ using \Cref{lem:bound-divisors}. This already gives us the desired bound, and we argue that the contribution of the rest of the segments is negligible.

We now wish to bound the total number of square-free divisors with $\omega(d) \ge k$ in the $k$-th interval, for $k \ge 3 \log_2 \log_2 N$. \Cref{lem:bound-num-div-on-high-omega} gives a bound of $O\parens{\frac{\Delta N \log ^ 3 N}{2^k}}$, which diminishes exponentially with $k$, so the total number of critical divisors for segments $k \ge 3 \log_2 \log_2 N$ is bounded by
\[ 
  O\parens{\frac{\Delta N \log ^ 3 N}{2^{3 \log_2 \log N}}} =
  O\parens{\Delta N}
\]

The bound we obtained for the contribution of segments $k \ge 3 \log \log N$ is much smaller than the desired bound, so the result follows.
\end{proof}

\begin{lemma}\label{lem:error-correction-logn}
Error correction can be done in $O(\Delta N \log N \log \log N)$ time.
\end{lemma}
\begin{proof}
By combining \Cref{lem:critical-div-iter} with
\Cref{lem:total-critical-divisors} and \Cref{eq:omega-sum-critical-interval}.
\end{proof}

\subsection{Reducing the sieve work}\label{sec:sieve-less}
Here we follow a path similar to \Cref{sec:fewer-div}, aiming at reducing the sieve work. We start by noting that we only need to factorize all numbers in the first $O(\log \log N)$ segments of the critical interval. In the remaining segments we ideally only need to factorize a small portion of the numbers -- only those with critical divisors. We find and factorize those numbers by restricting the sieve to numbers that have a divisor with many prime factors. We give the full details in the rest of this section.

\subsubsection{Numbers with many divisors are rare}
We provide the following lemma, that will be useful later.

\begin{lemma}\label{lem:bound-many-divisors}
For any $\epsilon>0$, if $N^{\epsilon} < S < N$, then amount of numbers in $(N, N+S]$ having at least $k$ distinct prime factors is $O\parens{\frac{S \log N}{2^k}}$
\end{lemma}
\begin{proof}
We apply \Cref{lem:sum-small-interval-bound} to the function $2^{\omega(n)}$. We obtain:
\begin{align*}
\sum_{n=N+1}^{N+S} 2^{\omega(n)} 
&= O\parens{\frac{S}{\log N} \exp\parens{\sum_{p\le 2N} \frac{2}{p}}} \\ 
&= O\parens{\frac{S}{\log N} \exp\parens{2 \ln \ln (2N) + O(1)}}
\\ &= O\parens{S \log N}
\end{align*}

We note that the indicator for $n$ having at least $k$ distinct prime factors is bounded by $2 ^ {\omega(n) - k}$. The desired bound is obtained by combining the above results.
\end{proof}

\subsubsection{Sieving a small interval}\label{sec:sqrt3-sieve}
The sieve of Eratosthenes can be improved to efficiently factorize numbers in a given range $[N, N+A]$ assuming $A = \Omega (\sqrt[3]{N})$. The algorithm is described in \cite{improved_eratosthenes} and is based on an earlier algorithm by \cite{improved_atkin} that implements a similar idea for the sieve of Atkin. Note that this sieve requires $O(A \log N)$ time rather than the $O(A \log \log N)$ time required by the classical sieve of Eratosthenes.

\subsubsection{Sieve restrictions}
Recall that the critical interval $(N, N+S]$ consists of segments of size $O(\Delta N)$, where in each segment the value of $\kbar(n)$ is constant. In the $k$-th segment, only numbers with $\omega(n) \ge k$ can contribute to error correction.

For $k < 18 \log_2 \log_2 N$, we use a regular sieve to factorize all numbers. This takes $O\parens{\Delta N (\log \log N)^2}$ time, since factorization sieve costs $\log \log N$ per element.

For the remaining segments $k \ge 18 \log_2 \log_2 N$, we note that only numbers with $\omega(n) \ge k$ may have critical divisors. For such numbers, multiplying their smallest $k/6$ distinct prime factors yields a divisor which is not larger than $N^{1/6}$. Therefore, we can find all these numbers by restricting our sieve to numbers of the form $n = md$, where $m\in\NN$ and $d$ iterates over all $d \le N^{1/6}$ with $\omega(d) \ge k/6$. We note that this method might recover the same number more than once. We can use a hash-table to remove duplicates, or we can just use the factorization obtained by the sieve in order to discard numbers obtained from $d$ which is not the product of the smallest $k/6$ distinct prime factors of the number (this product is unique per number).

For each segment of size $s$ and for each $d$ we have to sieve an interval of size $\ceil*{\frac{s}{d}}$. Since $s = \Theta(\Delta N) = \tilde{\Theta}(\sqrt N)$ and $d \le N^{1/6}$, the sieve interval will be at least $\tilde{\Omega}(\sqrt[3] N)$, large enough to satisfy the requirement of \Cref{sec:sqrt3-sieve}.

The cost of sieving a single such $d$ using \Cref{sec:sqrt3-sieve} is $O\parens{\frac{\Delta N}{d} \log N}$.  Omitting the big-$O$ notation and keeping the constraint $\omega(d) \ge k/6$ implicit to ease notation, we wish to bound:
\begin{align}
\sum_{d=1}^{N^{1/6}} \frac{\Delta N}{d} \log N
&\le \sum_{m=0}^{\floor*{\frac{1}{6}\log_2 N}} \sum_{d = 2^m}^{2^{m+1}-1} \frac{\Delta N}{d} \log N \nonumber \\ 
&\le \sum_{m=0}^{\floor*{\frac{1}{6}\log_2 N}} \sum_{d = 2^m}^{2^{m+1}-1} \frac{\Delta N}{2^m}  \log N\label{eq:sieve-restrictions-bound}
\end{align}

According to \Cref{lem:bound-many-divisors}, the amount of $d \in [2^m, 2^{m+1}-1)$ with $\omega(d) \ge k/6$ is $O\parens{\frac{2^m \log N}{2^{k/6}}}$. Substituting into \Cref{eq:sieve-restrictions-bound} we obtain a bound on the work of the $k$-th segment:
\begin{align*}
\sum_{m=0}^{\floor*{\frac{1}{6}\log_2 N}} \frac{\Delta N \log N}{2^{k/6}} \log N = \frac{\Delta N \log^3 N}{2^{k/6}}
\end{align*}

Thus the sieving work diminishes exponentially with the segment index $k$. We recall that we only apply a restricted sieve for segments with $k \ge 18 \log_2 \log_2 N$, thus bounding the sieve time by $O\parens{\Delta N}$. This  bound is already smaller than the time invested for sieving the first $18 \log_2 \log_2 N$ segments, which was $O\parens{\Delta N (\log \log N)^2}$.

This proves:
\begin{lemma}\label{lem:small-sieve}
Sieving can be done in $O\parens{\Delta N (\log \log N)^2}$ time.
\end{lemma}

\subsection{Accounting for divisors with a look-up table}\label{sec:div-table}

\subsubsection{Simplified error-correction condition}
In this section we present an improvement that further reduces the time complexity of the error correction phase by a logarithmic factor.
We are able to prove this improvement only assuming a randomized algorithm. We believe that in practice this assumption should work even without randomization, though we cannot prove it to work for a deterministic algorithm.

Recall that a divisor $d$ of $n\in (N, N+S]$ is relevant for error correction only if $\kbar (n/d)+\khat (d)\le \kbar(N)$ and $d$ a square-free with $\pmax(d)\le \sqrt N$, and in this case we accumulate $(-1)^{\omega(d)}$ into the error term.

Denote by $\{x\}=x-\floor*{x}$ the fractional part of $x$. 

Setting $d=\prod_{i} p_i$, we can rewrite the inequality $\kbar (n/d)+\khat (d)\le \kbar(N)$ as

\[ \floor*{\frac{\log_2 (n/d)}{\Delta}} + \sum_i \floor*{\frac{\log_2 p_i}{\Delta}} \le \kbar(N) \]

which in turn is equivalent to

\[ \frac{\log_2 (n/d)}{\Delta} - \left\{\frac{\log_2 (n/d)}{\Delta}\right\} + \sum_i \frac{\log_2 p_i}{\Delta} - \sum_i \left\{\frac{\log_2 p_i}{\Delta}\right\} \le \kbar(N) \]

or:

\[ \left\{\frac{\log_2 (n/d)}{\Delta}\right\} + \sum_i \left\{\frac{\log_2 p_i}{\Delta}\right\} \ge \frac{\log_2 n}{\Delta} - \kbar(N) \]

which is equivalent to:

\[ -\left\{\frac{\log_2 n}{\Delta}\right\} + \left\{\frac{\log_2 (n/d)}{\Delta}\right\} + \sum_i \left\{\frac{\log_2 p_i}{\Delta}\right\} \ge \kbar(n) - \kbar(N) \]

The left-hand side is an integer, being the sum of fractional parts of numbers whose sum is an integer. Since the right-hand side is an integer, and since $\left\{\frac{\log_2 (n/d)}{\Delta}\right\}$ is in $[0, 1)$, this inequality is equivalent to

\begin{equation}\label{eq:table-div-criterion} 
 -\left\{\frac{\log_2 n}{\Delta}\right\} + \sum_i \left\{\frac{\log_2 p_i}{\Delta}\right\} > \kbar(n) - \kbar(N) - 1
\end{equation}

The crux of this inequality is that it depends on $n$ only through $\kbar(n)$, $\left\{\frac{\log_2 n}{\Delta}\right\}$ and the fractional values of the $\frac{\log_2 p}{\Delta}$ for primes $p$ in $n$'s factorization. Since we only need to iterate over square-free divisors of $n$, we only need to keep these fractional values for the distinct primes in $n$. We observe that the correction terms are $(-1)^{\omega(d)}$ for the $d$'s satisfying the inequality, and hence also does not depend on the exact values of the primes.

\subsubsection{Look-up tables}\label{sec:lut}

Based on \Cref{eq:table-div-criterion}, the idea is to construct a look-up table based on approximations of the prime divisors.

More concretely, we set some new precision parameter $\epsilon$, and compute $\floor*{\frac{1}{\epsilon}\left\{\frac{\log_2 n}{\Delta}\right\}}$ and the multi-set $\left\{\floor*{\frac{1}{\epsilon}\left\{\frac{\log_2 p}{\Delta}\right\}} : p\mid n \right\}$ of each $n$ in the critical interval, where each prime factor is taken once, ignoring multiplicities. The total time required for this operation is that of iterating through the factorization of each number in the critical interval. This costs $\omega(n)$ per number, and thus it is on par with the sieve work, already bounded in the previous section. The goal is to use these rounded values as the key to the look-up table.

For each segment in the critical interval we construct a different look-up table that will be shortly described. The important idea is to restrict the number of factors: for each segment, we only create a table for numbers $n$ with up to $4 \log_2 \log_2 N$ prime factors including multiplicity, since these numbers are the majority. This also means there is no use for constructing these tables for more than the first $\approx 4 \log_2 \log_2 N$ segments, as the critical divisors for numbers in these segments will not be accounted for.

Hence, for each such segment, there are 
\begin{equation}\label{eq:table-size-general}
\text{Table size} = \binom{4 \log_2 \log_2 N + \floor*{1/\epsilon} + 1}{4 \log_2 \log_2 N} \cdot \parens{\floor*{1/\epsilon} + 1}
\end{equation} possible entries to be computed in the table -- the number of ways to distribute up to $4 \log_2 \log_2 N$ balls (the factors) to the $\floor*{1/\epsilon} + 1$ cells (according to the values taken by $\floor*{\frac{1}{\epsilon}\left\{\frac{\log_2 p}{\Delta}\right\}}$. The number of entries is multiplied by $\floor*{1/\epsilon} + 1$ in order to take into account the value of $\floor*{\frac{1}{\epsilon}\left\{\frac{\log_2 n}{\Delta}\right\}}$).

In each entry in the table, we add all error-terms corresponding to divisors for which the approximated factorization is enough to determine. We additionally store the description of any divisor not entirely determined by the given approximation, as an additional list in the look-up table entry. We will choose $\epsilon$ to make this list short enough in expectation.

For each specific pair $(n, d)$ where $d$ is a square-free divisor of $n$, let us inspect \Cref{eq:table-div-criterion} again, and note that each summand has an additive uncertainty of at most $\epsilon$, and hence the left-hand side has an uncertainty bounded by $4 \epsilon \log_2 \log_2 N$. We now analyze the probability of the inequality not being fully determined by the given approximation, when $\frac{1}{\Delta}$ is chosen uniformly at random from some range $\left[X, X+1\right]$ for some $X$ (since we will eventually choose $\frac{1}{\Delta} = \tilde{O}(\sqrt N)$ in order to balance the two parts of our algorithm, we are free to randomize it in this way).

Since the right-hand side of \Cref{eq:table-div-criterion} is an integer, if the left-hand side is far enough from an integer we can check if the inequality holds given the approximated values. The fractional value of the left-hand side is (using the general $\{\{x\}\pm\{y\}\}=\{x\pm y\}$):

\[
\left\{-\frac{\log_2 n}{\Delta} + \sum_i \frac{\log_2 p_i}{\Delta}\right\} = \left\{\frac{1}{\Delta}\parens{-\log_2 n + \sum_i \log_2 p_i}\right\}
\]

We now observe that, if $\Delta$ is randomly chosen such that $\frac{1}{\Delta}$ is uniformly distributed in $[X,X+1]$, the value of $\frac{1}{\Delta}\parens{-\log_2 n + \sum_i \log_2 p_i}$ is uniformly distributed in $\left[-(X+1)\log_2 \frac{n}{d}, -X\log_2 \frac{n}{d}\right]$, which is an interval of length at least $1$ as long as $d<n$, since then $n/d\ge 2$. Therefore, the probability of the fractional value being at most $\epsilon\cdot 4 \log_2 \log_2 N$ from an integer is bounded by $O\parens{\epsilon \log \log N}$. The special case of $d=n$, to which the above logic does not apply, can be checked without using the table.

In total, then, using \Cref{lem:sum-small-interval-bound} once again, we can bound the expected number of divisors (of numbers $n$ with at most $\omega(n)\le 4 \log_2 \log_2 N$ divisors) not being accounted for in each segment (of size $O(\Delta N)$) by 

\[O\parens{\epsilon \log \log N \cdot \sum_{n\text{ in segment}} 2^{\omega(n)}} = O\parens{\epsilon \Delta N \log N \log \log N}
\]

If we set $\epsilon = O\parens{\frac{1}{\log^2 N \log \log N}}$, the total work done accessing the remaining undetermined divisors will be negligible compared to accessing the entry of each number.

Using \Cref{eq:table-size-general} with the general estimation $\binom{n}{k}=\parens{\Theta \parens{\frac{n}{k}}}^k$, this implies a table of size $\exp(O(\log^2 \log N))$. This is sub-polynomial in $N$, and therefore is asymptotically negligible.

\subsubsection{Work on numbers missing from table}

We use \Cref{lem:bound-num-div-on-high-omega} to bound the work on numbers missing the table, that is, $n$ with $\omega(n) > 4 \log_2 \log_2 N$. The total extra work is then bounded by the work needed to iterate over all square-free divisors of numbers $n$ with $\omega(n)>\log\log N$, which by \Cref{lem:bound-num-div-on-high-omega} is bounded by

\[ O\parens{2^{-4 \log_2 \log_2 N} S \log^3 N} = O\parens{\frac{S}{\log N}}. \]

Using \Cref{lem:smaller-critical-interval}, this is bounded by $O\parens{\Delta N}$ which is indeed negligible compared to the rest of the error correction phase work.

\begin{lemma}\label{lem:ec-using-tables}
Error correction can be done in $O\parens{\Delta N (\log \log N)^2}$ time in expectation.
\end{lemma}
\begin{proof}
We construct a look-up table for numbers $n$ with $\omega(n) \le 4 \log_2 \log_2 N$, as described. We have seen that the total work for creating the table is negligible, since its size is sub-polynomial (and hence the table can be created by a simple brute-force). We iterate only on numbers obtained by the sieve, which in \Cref{sec:sieve-less} was bounded by $O(\Delta N \log \log N)$, and for each such number we use inspect the relevant entry in the look-up table in $O(\log \log N)$ time (the time required for computing the corresponding key). We have already seen that the remaining work on the undetermined divisors is negligible, as well as the work on numbers with more than $4 \log_2 \log_2 N$ prime factors, and hence the total error correction time is dominated by $O\parens{\Delta N \log \log N)} \cdot O\parens{\log \log N}$.
\end{proof}

\subsection{Optimized complexity}\label{sec:complexity}

\begin{theorem}\label{thm:pi-logn}
$\pi(N)$ can be computed in $O\parens{\sqrt N \log N (\log \log N)^{3/2}}$ time in expectation.
\end{theorem}
\begin{proof}
We gather the FFT time given in \Cref{lem:hmu-log2}, the divisor iteration time of \Cref{lem:ec-using-tables} and the sieving time of \Cref{lem:small-sieve} to obtain the total time complexity of
\begin{equation}
O\parens{\sqrt N \log \log N + \frac{\log^2 N \log \log N}{\Delta} + \Delta N (\log\log N)^2}
\end{equation}
This is minimized for $\Delta = \Theta \parens{ \frac{\log N}{\sqrt {N\log\log N}} }$ obtaining the required total time complexity.
\end{proof}

Note that the randomness needed for the guarantees of \Cref{lem:ec-using-tables} for using the look-up tables may be avoided by using \Cref{lem:error-correction-logn} instead for divisor iteration. This yields a deterministic time complexity of $O\parens{\Delta N (\log N)^{3/2} \log \log N}$.

\subsection{Algorithm summary}

We summarize here the algorithm described up to this point. We first find the primes up to $\sqrt N$ using a sieve, and then partition them into $\log \log N$ disjoint subsets according to their size, such that the $m$-th subset ($1\le m \le \log_2 \log_2 N$) includes primes $p$ with $\frac{\log_2 N}{2^{m+1}} \le \log_2 p < \frac{\log_2 N}{2^{m}}$ (see \Cref{sec:partition-primes}). In each subset,  we round all primes in log-scale to multiples of $\Delta = \Theta \parens{ \frac{\log N}{\sqrt {N\log\log N}} }$ (which was chosen as the inverse of a number drawn uniformly from an interval of length 1, to enjoy the proven guarantees of \Cref{sec:lut}), and count the number of primes in each such logarithmic interval to form an array of size $O\parens{\frac{r_\text{max} \log p_\text{max}}{\Delta}} = O\parens{ \sqrt{N \log \log N}}$. We apply FFT to each of these $\log \log N$ arrays, and from there we are able to compute, one Fourier coefficient at at time, the $\widetilde{C_r}$ arrays for all relevant values of $r$. This is done by the exponentiation-of-power-series method as described in \Cref{sec:better-better-newton}.

For each subset of primes, we combine its $\widetilde{C_r}$'s with alternating signs (using \Cref{eq:mu-hat-newton}) to form a partial $\hat{\mu}$ for these primes.

We then perform an inverse FFT on each of these $\log_2 \log_2 N$ results, truncate the resulting arrays at $\kbar(N)=O\parens{\sqrt{N \log \log N}}$, and convolve to obtain $\hmu$.

We then convolve $\hmu$ and $\bone$ (the array that counts the number of integers in each segment), and sum the entries of the resulting array up to $\kbar(N)$. We now almost have our result, and only need to cancel out the error term described in \Cref{thm:error-term}.

Before proceeding to the error correction phase we construct a table of size $\exp(O((\log\log N)^2))$ for each segment (construction time is similar to the size of the table), accounting for all divisors of numbers $n$ with $\omega(n)\le 4 \log_2 \log_2 N$ (\Cref{sec:div-table}). We then sieve numbers $n$ in the first $18 \log_2 \log_2 N$ segments: $\kbar(n) \in [\kbar(N), \kbar(N) + 18 \log_2 \log_2 N$]. The rest of the critical interval is sieved segment-by-segment, where in each segment the value of $k = \kbar(n)- \kbar(N)$ is constant. In each such segment we sieve numbers that are multiples of $d$ for all $d \le N ^ {1/6}$ with $\omega(d) \ge k/6$. We thus obtained the factorization of all numbers that may contribute to the error term.

For each number with $\omega(n)\le 4 \log_2 \log_2 N$, we query the look-up table using its factorization, remove the output value from the result, and iterate over the remaining undetermined divisors (which are also part of the table's output). For each such undetermined divisor $d$ of $n$, we check if $\kbar(n/d) + \khat(d) \le \kbar(N)$, if $\pmax(d)\le\sqrt N$, and if $d$ is square-free. If all these conditions are satisfied, we remove $\mu_{\le \sqrt N}(d)=(-1)^{\omega(d)}$ from the result. For numbers with $\omega(n) > 4 \log_2 \log_2 N$ we skip the table and treat all divisors as undetermined.

Finally, we add $\pi(\sqrt N) - 1$ to obtain $\pi(N)$ according to \Cref{lem:pi-relation}.

\section{Space Improvements}\label{sec:space-improvement}
The space complexity of the algorithm described so far is $\tilde{O}(\sqrt N)$. There are two such memory requirements:
\begin{enumerate}
  \item We apply FFT on arrays of size
  $O \parens{ \frac{\log N}{\Delta} } = 
  \tilde{O} \parens{ \sqrt N }$.
  \item We sieve a segment of size
  $S=O\parens{\Delta N \frac{\log N}{\log \log N}} =
  \tilde{O} \parens{ \sqrt N }$. 
\end{enumerate}

In this section we explain how to reduce the space complexity from $\tilde{O}(\sqrt N)$ to $\tilde{O}(\sqrt[3]{N})$ by addressing each of the above issues. We are aware that some of the improvements presented in this section may be incompatible with time improvements presented in \Cref{sec:time-improvement}. However, for larger values of $N$ we assume that the reduction in memory will be preferable even at the cost of logarithmic factors in time. Thus, we ignore logarithmic factors in this section.

\subsection{Using fewer exact primes}
Convolving huge arrays is essential to our algorithm. Our idea on how to avoid storing those huge arrays starts with replacing $\hmu$ with $\hmt$. We will explain in the next subsection how this improves memory. Before that, let us explain how we can obtain the desired result using $\hmt$ rather than $\hmu$.

Originally, we used $\bone\Conv\hmu$ to count numbers $\le N$ with no prime factors smaller than $\sqrt N$, up to segmentation errors. Instead, $\bone\Conv\hmt$ corresponds to numbers with no prime factors smaller than $\sqrt[3] N$. These include the desired prime numbers up to $N$, but also products of two prime numbers larger than $\sqrt[3] N$, as well as the number $1$. The number $1$ is easy to remove by computing $A = \bone\Conv\hmt - \delta_0$ (recall that $\delta_k$ is the array with zeros everywhere except for a single $1$ at the $k$-th cell). We are left with the task of removing products of two primes from $A$. 

The trick is to convolve $A$ with itself, truncating at $\kbar(N)$. We obtain the multiset of products $\{\{ \kbar(p_i p_j) \}\}$ for $p_i$ and $p_j$ larger than $\sqrt[3] N$. Thus $A \Conv A$ counts twice each product of two different primes, and once each square of a prime. By computing $A - \frac{1}{2} (A \Conv A)$ and summing up to $\kbar(N)$ we obtain the following contributions:

\begin{enumerate}
  \item Prime numbers in $\left( \sqrt[3] N, N \right]$, each contributing $1$ to the sum. This is the desired result, up to primes $\le \sqrt[3] N$ that were already found.
  \item Squares of primes in $\left( \sqrt[3] N, \sqrt N \right]$, each contributing $1$ to $A$ and $1$ to $A \Conv A$, totaling for $\frac{1}{2}$ the number of primes in this range. This contribution should be canceled by counting the number of primes up to $\sqrt N$.
  \item Segmentation errors, since we use our approximate array convolution technique rather than the accurate Dirichlet convolution.
\end{enumerate}

We note that as before, we need an error correction phase where wrongly accumulated contributions are canceled. Since we now have a term
\[ \bone\Conv\bone\Conv\hmt\Conv\hmt , \]
we will have to iterate through factorization of each number $n=d_1 d_2 d_3 d_4$ and subtract $\one(d_1) \one(d_2) \mu_{\le \sqrt[3] N}(d_3) \mu_{\le \sqrt[3] N}(d_4)$ if it was accumulated, that is if $\kbar(d_1) + \kbar(d_2) + \khat(d_3) + \khat(d_4) \le \kbar(N)$. This can be done without actually iterating all factorizations, but as mentioned, we ignore factors of $\log N$ in the complexity at this point. Hence, a direct application of \Cref{lem:sum-small-interval-bound} is enough to bound the required time.

Generalizing, we have the following identity, which we will later use to compute $\pi(N)$ with space complexity $\tilde{O}(N^{1/4})$.

\begin{lemma}\label{lem:ln-identity}
Let $t\ge 1$ be a natural number, and let $f=(\one\Conv \mu_{\le N^{1/t}}) - \delta_1$. Denote by $f^{\Conv k}$ the $k$-th Dirichlet convolution of $f$ with itself. That is, $f^{\Conv 1}=f$ and $f^{\Conv (k+1)}=f^{\Conv k}\Conv f$.
Then:

\[ \pi(N) = \sum_{n=1}^{N} \sum_{k=1}^{t-1} \frac{(-1)^{k-1}}{k}f^{\Conv k}(n) + \pi(N^{1/t}) - \sum_{k=2}^{t-1} \frac{1}{k} \parens{\pi(N^{1/k})  - \pi(N^{1/t})}\]
\end{lemma}
\begin{proof}
The function $f$ is the indicator function of integers $>1$ whose prime factorization contains only primes larger than $N^{1/t}$. 

For a function $g$ we denote by $D_g$ the Dirichlet series of $g$:

\[ D_g(s)=\sum_{n=1}^\infty \frac{g(n)}{n^s}\]

Recall that $D_{g_1 \Conv g_2} = D_{g_1}\cdot D_{g_2}$ for any two functions $g_1, g_2$.

For our $f$, we have $D_f(s)=\sum_{n>1} \frac{1}{n^s}$, where $n$ iterates over numbers divisible only by primes above $N^{1/t}$. Observe that

\[ D_f(s) = -1 + \prod_{p>N^{1/t}} \frac{1}{1-p^{-s}} \]

Let $g=\sum_{k=1}^{\infty} \frac{(-1)^{k-1}}{k}f^{\Conv k}$. For $n \le N$ it coincides with $\sum_{k=1}^{t-1} \frac{(-1)^{k-1}}{k}f^{\Conv k}$, since $f^{\Conv t}$ is zero at all integers $\le N$. We have:

\begin{align*}
D_g(s) &= \sum_{k=1}^\infty \frac{(-1)^{k-1}}{k} D_f(s)^k = \ln\parens{1+D_f} = \ln\parens{\prod_{p>N^{1/t}} \frac{1}{1-p^{-s}}} \\ &= -\sum_{p>N^{1/t}} \ln (1-p^{-s}) = \sum_{p>N^{1/t}} \sum_{k=1}^\infty \frac{1}{k}p^{-sk}
\end{align*}

That is, $g$ is supported on prime powers of all primes above $N^{1/t}$, with $g(p^k)=\frac{1}{k}$ on such prime powers.

Hence,

\[ \sum_{n=1}^N g(n) = \sum_{k=1}^\infty \frac{1}{k}\sum_{N^{1/t}<p\le N^{1/k}} 1 = \sum_{k=1}^{t-1} \frac{\pi(N^{1/k})-\pi(N^{1/t})}{k} \]

The lemma's statement readily follows.
\end{proof}

As a side note, observe that in the limit of $t\to\infty$, \Cref{lem:ln-identity} gives:

\begin{equation}\label{eq:ln-identity-limit}  \sum_{n=1}^{N} \parens{\sum_{k=1}^{\infty} \frac{(-1)^{k-1}}{k}(\one - \delta_1)^{\Conv k}}(n) = \sum_{k=1}^{\infty} \frac{\pi(N^{1/k})}{k} .
\end{equation}

This is equivalent (after a Mellin transform) to Riemann's identity $\frac{\ln \zeta(s)}{s} = \int_0^\infty \pi^*(x)x^{-s-1}dx$ with $\pi^*(x)$ defined as $\pi^*(x)=\sum_{k=1}^\infty \frac{\pi(x^{1/k})}{k}$.

\subsection{Working only in Fourier space}

As described, the general idea is to compute $\sum_{k=0}^{\kbar(N)} (A - \frac{1}{2}(A*A))[k]$, then combine with knowledge on primes $\le \sqrt N$ and an error correction phase, to finally obtain $\pi(N)$. But we still need to explain how to compute this sum without storing the whole array $A$ in memory at once.

The idea is to work only in Fourier space. Denote by $\zeta_L$ a primitive root of unity of order $L$. For an array $B$ of length $L$, we have its Fourier transform $\tilde{B}$. Instead of computing the inverse Fourier transform
\[ B[k] = \frac{1}{L} \sum_{\ell=0}^{L-1} \tilde{B}[\ell]\zeta_L^{k\ell} \]
we compute the output sum directly
\[ L \sum_{k=0}^{n} B[k] = \sum_{k=0}^{n} \sum_{\ell=0}^{L-1} \tilde{B}[\ell]\zeta_L^{k\ell} \]

\[ = \sum_{\ell=0}^{L-1} \tilde{B}[\ell] \sum_{k=0}^{n} \zeta_L^{k\ell}= (n+1)\tilde{B}[0] + \sum_{\ell=1}^{L-1} \frac{\zeta_L^{(n+1) \ell} - 1}{\zeta_L^\ell - 1}  \tilde{B}[\ell]. \]

That is, we have shown:

\begin{lemma}\label{lem:prefix-sum-from-fourier}
For an array $B$ of length $L$ and its Fourier transform $\tilde{B}$, the following holds:
\[  L \sum_{k=0}^{n} B[k] = (n+1)\tilde{B}[0] + \sum_{\ell=1}^{L-1} \frac{\zeta_L^{(n+1)\ell} - 1}{\zeta_L^\ell - 1}  \tilde{B}[\ell] . \]
\end{lemma}

We can apply this lemma for any convolution we want to sum up to $\kbar(N)$, for example:

\begin{align*} & L \sum_{k=0}^{\kbar(N)} \parens{A - \frac{1}{2}(A*A)}[k] \\&= (\kbar(N)+1)\parens{\tilde{A}[0] - \frac{1}{2}\tilde{A}[0]^2} + \sum_{\ell=1}^{L-1} \frac{\zeta_L^{(\kbar(N)+1)\ell} - 1}{\zeta_L^\ell - 1} \parens{\tilde{A}[\ell] - \frac{1}{2}\tilde{A}[\ell]^2},
\end{align*}
where $L$ is the size of the arrays used.

Thus, as long as we are able to compute all Fourier coefficients $\tilde{A}[\ell]$ without much memory (which necessarily means not storing them all at once), we are still able to compute the approximate count, and in the same time and space complexity as for computing the $\tilde{A}[\ell]$ values.

Recall that $A = \bone \Conv \hat\mu_{\le N^{1/t}} - \delta_0$. Since $\tilde{\delta}_0$ is a constant array (evaluations of a constant polynomial), in order to compute $\tilde{A}[\ell]$ we only need to compute $\tilde{\bone}[\ell]$ and $\tilde{\hat{\mu}}_{\le N^{1/t}}[\ell]$ in the same index $\ell$.

As it turns out, we are able to compute the entries $\tilde{\hat{\mu}}_{\le N^{1/t}}[\ell]$ in $\tilde{O}(N^{1/t})$ space and $\tilde{O}(\sqrt N)$ time. However, reducing the space consumption of computing the entries $\tilde{\bone}[\ell]$ is more challenging, and is the reason we are not able to retain the $\tilde{O}(\sqrt N)$ time complexity for lower space complexities.

\subsection{Smaller FFTs for Möbius}\label{sec:small-fft-for-mobius}

The previous subsections showed that $\hat\mu_{\le N^{1/t}}$ suffices for our purposes. We now explain how we can compute it as well as the desired result with $\tilde{O}(N ^ {1/t})$ space.

We would like to work with the approach of a single FFT. That is, we do not use the prime partitioning improvement, and we pad the FFT arrays enough so we can apply Newton's identities entirely in the Fourier space. This way, the entire computation can be carried on each Fourier coefficient independently with $O(r_\text{max})$ memory. We are left with the problem of entering the Fourier space, which naively requires storing the entire array.

Our algorithm requires the Fourier transform of the array that corresponds to primes $\le N^{1/t}$. In line with the previous sections, we denote this array and its Fourier transform by $E_1$ and $\widetilde{E_1}$. We note that by definition:
 \[ E_1 = \sum_{p \le N^{1/t}} \delta_{\kbar(p)} \]

Observe that this array is very sparse. We now explain how to sequentially compute the Fourier transform of an array using memory proportional to its sparsity.

First, we let $L$ be the size of the array (which, as we recall, is on the order of $\tilde{\Theta}(\sqrt N)$). Write $L = L_0 L_1$ where $L_1$ is approximately equal to the sparsity of the input array, that is $L_1$ is approximately the number of primes $\le N^{1/t}$. Note that for FFT $L$ is usually a power of two, so such $L_1$ can be chosen up to a factor of two. Hence, we assume we can take $L_1=\tilde{\Theta}(N^{1/t})$.

\begin{lemma}\label{lem:low-mem-E1}
Given $L=L_0L_1$ where $L_1=\tilde{\Theta}(N^{1/t})$, and a value $0\le \ell_0 \le L_0 - 1$, it is possible to compute all values $\widetilde{E_1}[\ell]$ satisfying $\ell\equiv 
 \ell_0 \pmod{L_0}$, in $\tilde{O}(N^{1/t})$ time and $\tilde{O}(N^{1/t})$ space.
\end{lemma}

\begin{proof}
Let $\zeta_L$ be the fundamental root of unity of order $L$ used for FFT. The value of the Fourier transform at index $\ell$ is:
 \[ 
    \widetilde{E_1}[\ell] = 
    \sum_k E_1[k] \zeta_L ^ {\ell k} = 
    \sum_p \zeta_L ^ {\ell \kbar(p)}
  \]

Now write the index in the Fourier space as $\ell = \ell_0 + L_0\ell_1$ where $\ell,\ell_0,\ell_1$ are all non-negative integers and in addition $\ell < L, \ell_0 < L_0$ and $\ell_1 < L_1$. We argue that we can compute the Fourier coefficient at all indices with the same $\ell_0$ using an FFT of size $L_1$. To do so, fix $\ell_0$ and notice:
\[ 
\widetilde{E_1}[\ell_0 + L_0\ell_1] =  
\sum_p \zeta_L ^ {(\ell_0 + L_0\ell_1) \kbar(p)} = 
\sum_p \zeta_L ^ {\ell_0 \kbar(p)} \parens{ \zeta_L ^ {L_0} } ^ {\ell_1 \kbar(p)}
\]
Since $\ell_0$ is fixed, the final summation has the form of a Fourier transform of size $L_1$, since $\zeta_L ^ {L_0}$ is a root of unity of order $L_1$. Explicitly:

\begin{equation}\label{eq:low-mem-fft-formula-mobius}
\widetilde{E_1}[\ell_0 + L_0\ell_1] =
\sum_{m_1=0}^{L_1-1}  \parens{ \zeta_L ^ {L_0} } ^ {\ell_1 m_1}\sum_{\kbar(p)\equiv m_1 (\text{mod } L_1)} \zeta_L ^ {\ell_0 \kbar(p)}
\end{equation}

Hence, the algorithm is as follows. We partition the primes $\le N^{1/t}$ by the value of $\kbar(p) \bmod L_1$ and compute $\zeta_L ^ {\kbar(p)}$ for each prime. Then, for each $\ell_0$ up to $L_0$ construct an array of size $L_1$ where the $\ell_1$-th cell equals $\sum \zeta_L ^ {\ell_0 \kbar(p)}$, summing over primes with $\ell_1 \equiv \kbar(p) \pmod {L_1}$. This array can be constructed with time and space proportional to the number of primes. Then, apply FFT of size $L_1$ on this array. The result is the Fourier transform at indices $\ell  \equiv \ell_0 \pmod {L_0}$.
\end{proof}

As a corollary, we have:

\begin{lemma}\label{lem:low-mem-mobius}
Given $L=L_0L_1$ where $L_1=\tilde{\Theta}(N^{1/t})$, and a value $0\le \ell_0 \le L_0 - 1$, it is possible to compute all values $\widetilde{\hat{\mu}}_{\le N^{1/t}} [\ell]$ for $\ell\equiv 
 \ell_0 \pmod{L_0}$, in $\tilde{O}(N^{1/t})$ time and $\tilde{O}(N^{1/t})$ space.
\end{lemma}
\begin{proof}
As was done in \Cref{sec:better-newton} use \Cref{eq:newton-like-fourier} to compute all values of $\widetilde{C_r}$ in the relevant indices by first computing them for $\widetilde{E_r}$. Since all computations are element-wise in the Fourier space, it is indeed enough to compute only the subset of values for the $\widetilde{E_r}$ arrays. We can compute all $\widetilde{E_r}$ in a similar way to \Cref{lem:low-mem-E1} for the $O(\log N)$ relevant values of $r$.
\end{proof}

Next, we show that given $\ell_0$, we are also able to efficiently compute the values of $\widetilde{\bone}[\ell]$ in the same indices with low memory.

\subsection{Handling the \texorpdfstring{$\bone$}{one} array}

\begin{lemma}[informal]\label{lem:bone-space-time}
Given space $M = \tilde{\Omega}(1/{\sqrt \Delta})$ such that $M\le \sqrt{N}$, the Fourier transform $\tilde{\bone}$ of $\bone$ can be computed (in batches) in time complexity 
\[
\tilde{O}\parens{\frac{1}{\Delta^{3/2} M^{1/2}}}
\]
\end{lemma}

In \Cref{sec:space-improvement-details} we sketch a proof of this lemma.

\subsection{Sieving small segments}\label{sec:sieve-small-segments}
For the error correction phase, we need to factorize the numbers in the interval $(N, N+S]$. Segmented sieve is usually used to reduce memory footprint, but since we must sieve primes up to $\sqrt{N + S}$, reducing the segment size below $\tilde{O}(\sqrt N)$ will affect time complexity.

We already mention in \Cref{sec:sqrt3-sieve} that \cite{improved_eratosthenes} offers an improved sieve algorithm that works with smaller segments, We apply this algorithm in order to factorize the numbers in the critical interval $(N, N+S]$ by iterating segments of size $\sqrt[3] N$. Note that reducing the memory footprint using \cite{improved_eratosthenes} adds an extra factor of $\log N$ to the time complexity.

It should be noted that, in theory, it is possible to complete the error correction phase in $O(S\cdot N^\epsilon)$ time and $\log^{O(1)}(N)$ space, by iterating over all numbers in the critical interval, factoring them with any sub-exponential factorization algorithms.

\subsection{Complexity analysis}
We assume space complexity $M$ and derive the time complexity by combining the methods presented so far.

\begin{theorem}\label{thm:pi-space-time-tradeoff}
Using $M$ memory, $N^{2/9} < M < N^{1/2}$, our algorithm computes $\pi(N)$ in $\tilde{O}\parens{\frac{N^{3/5 + \epsilon}}{M^{1/5}}}$ time for any fixed $\epsilon>0$. The corresponding optimal $\Delta$ is given by $\frac{1}{\Delta} = \tilde{\Theta}\parens{N^{2/5} M^{1/5}}$.

For $M \ge M^{1/3}$, the above time complexity can be achieved deterministically with $\epsilon = 0$.

In particular, one can compute $\pi(N)$ in:
\begin{itemize}
    \item $O\parens{N^{5/9+\epsilon}}$ time and $O\parens{N^{2/9+\epsilon}}$ space for any fixed $\epsilon>0$
    \item $\tilde{O}\parens{N^{8/15}}$ time and $\tilde{O}\parens{N^{1/3}}$ space
\end{itemize}
\end{theorem}

\begin{proof}
Using \Cref{sec:sieve-small-segments}, the error correction time is $\tilde{O}(S) = \tilde{O}(\Delta N)$. This can be done deterministically for $M\ge N^{1/3}$. With less memory, error correction can be done using non-deterministic methods in $O(\Delta N^{1+\epsilon})$ time.

The other bottleneck of the algorithm is computing the Fourier transform of the $\bone$ array, which can be done in $\tilde{O}\parens{\frac{1}{\Delta^{3/2} M^{1/2}}}$ according to \Cref{{lem:bone-space-time}}.

Choosing $\Delta =\tilde{\Theta}\parens{N^{-2/5} M^{-1/5}}$ balances these bottlenecks, yielding the desired time complexity.
\end{proof}

We note that space complexity can be reduced to $O(N^{r})$ for $r<2/9$ with a different time-space trade-off, but we omit the details.

\section{Computing other number-theoretic functions}
\label{sec:other-functions}

\subsection{Sum of multiplicative function evaluated at primes}

Let $h$ be a completely multiplicative function. That is, $h(ab)=h(a)h(b)$ for any $a,b\in\NN$. Informally, we require that $h$ can be summed efficiently over intervals of integers. We show how our methods can be extended to efficiently computing $\sum_{p\le N} h(p)$.

We briefly discuss the changes to the prime-counting algorithm required when changing the $\one$ function to a different function $h$.

\begin{itemize}
    \item $\mu_p$ for each prime $p$ is now replaced by the function $\eta_p$ such that $\eta_p(1)=1$ and $\eta_p(p)=-h(p)$, and 0 otherwise. Put differently, $\eta_p=\mu_p\cdot h$. It follows that
    
    \[ \eta_{\le \sqrt N}(n) = \Conv_{p\le \sqrt N} \eta_p = h(n)\cdot \mu_{\le \sqrt N}(n) \]

    \item \Cref{lem:pi-relation} is replaced by the corresponding
    \begin{equation}\label{eq:pi-relation-generalized}
    \sum_{n=1}^N (h \Conv \eta_{\le \sqrt N})(n) = 1 + \sum_{\sqrt{N} < p \le N} h(p)
    \end{equation}

    The proof of \Cref{eq:pi-relation-generalized} is straightforward by expanding:

    \[ \sum_{n=1}^N (h \Conv \eta_{\le \sqrt N})(n) = \sum_{n=1}^N \sum_{d\mid n} h(n/d) \eta_{\le \sqrt N}(d) = \sum_{n=1}^N h(n) \sum_{d\mid n} \mu_{\le \sqrt{N}}(d) \]

    \item The equivalent of \Cref{eq:hmu-convolution} is
    \begin{equation}\label{eq:hmu-convolution-generalized}
    \hat{\eta}_{\le \sqrt N} = \Conv_{p \le \sqrt N} \bar{\eta}_{p} = \Conv_{p \le \sqrt N} \parens{ \delta_{0} - h(p)\delta_{\kbar (p)} }
    \end{equation}

    from which the generalization of of \Cref{thm:hmu-val} also follows:
    \begin{equation}\label{eq:hmu-val-generalization}
    \hat{\eta}_{\le \sqrt{N}}[k] = \sum_{n:\,\khat(n)=k} \eta_{\le \sqrt N}(n)
    \end{equation}

    \item The Newton identities discussed in \Cref{sec:newton-identities} are also naturally generalized to use

    \begin{equation}
        C_r[k] = \sum_{\substack{n: \khat(n)=k \\ \omega(n)=r \\ \mu_{\le \sqrt{N}}(n) \neq 0}} h(n)
    \end{equation}

    and similarly:
    \begin{equation}
        E_r[k] = \sum_{p: \khat(p^r)=k} h(p^r)
    \end{equation}

    With these definitions, \Cref{eq:newton-like} remains valid as-is.

    \item As for the error correction formula developed in \Cref{sec:error-correction}, we have the generalized:

    \begin{equation}\label{eq:error-term-generalized}
    \sum_{k=0}^{\kbar (N)} (\bar{h}\Conv \hat{\eta}_{\le \sqrt{N}})[k] - \sum_{n=1}^N (h \Conv \eta_{\le \sqrt N})(n) = \smashoperator[r]{\sum_{\substack{d_1d_2>N \\ \kbar (d_1)+\khat (d_2)\le \kbar(N)}}} \; h(d_1d_2)\cdot \mu_{\le \sqrt N}(d_2)
    \end{equation}

    where we have used the identity $h(d_1)\cdot \eta_{\le \sqrt N}(d_2) = h(d_1d_2) \mu_{\le\sqrt{N}}(d_2)$ which follows from the multiplicativity of $h$ and the definition of $\eta$. This expression for the error term follows from an analogous derivation done in \Cref{sec:error-correction}.
\end{itemize} 

Some, but not all, of the time improvements discussed in \Cref{sec:time-improvement} are also applicable for this generalized problem. For brevity, we omit the details, and ignore logarithmic factors. We have thus shown the following:

\begin{theorem}\label{thm:sum-generalized}
Let $h$ be a completely multiplicative function such that $\sum_{n=1}^m h(n)$ can be evaluated in $O(T)$ time for any $m \le N$.
Then $\sum_{p\le N} h(p)$ can be evaluated in $\tilde{O}(T\sqrt N)$ time.
\end{theorem}

\subsection{Sum of primes}
An immediate corollary of \Cref{thm:sum-generalized} is that we can efficiently compute the sum of primes, and in fact any constant positive power of them:

\begin{corollary}\label{cor:sum-prime-powers}
For any integer $\ell\ge 0$, we can compute $\sum_{p\le N}p^{\ell}$ in $\tilde{O}\parens{\sqrt N}$ time.
\end{corollary}

\subsection{Counting primes in residue classes}
We now show how to efficiently count, for any $r, m$, the number of primes $p\le N$ such that $p\equiv r \pmod{m}$, denoted $\pi(N, m, r)$.

To this end, we employ \Cref{thm:sum-generalized}. As a simple example, consider the function 
\[ h(n)=\begin{cases} 0 , & n\text{ is even} \\ 1, & n\equiv 1\pmod{4} \\ -1, & n\equiv 3 \pmod{4}\end{cases} \]

$h$ is completely multiplicative and $\sum_{n=1}^m h(n)$ can be computed in $O(1)$ standard operations for any $m$. Using \Cref{thm:sum-generalized} we can compute $\pi(N, 4, 1) - \pi(N, 4, 3)$ in $\tilde{O}(\sqrt N)$ time. Since counting primes up to $N$ recovers $\pi(N, 4, 1) + \pi(N, 4, 3)$, we can extract the individual values of $\pi(N, 4, 1)$ and $\pi(N, 4, 3)$.

More generally, for any constant $m$ we can compute the sum $\sum_{p\le N} \chi_k(p)$, where $\chi_k: \ZZ \to \CC$ is any of the $\varphi(m)$ Dirichlet characters\footnote{Practically, our algorithms always work with integers modulo $p$ for an appropriate prime $p$ that enables fast Fourier transforms. In this case, we add the constraint that $p\equiv 1\pmod{\varphi(m)}$, so there exists a primitive $\varphi(m)$-th root of unity, to replace $\CC$ with $\FF_p$.} of modulus $m$. Then, letting $S_k(N) = \sum_{p\le N} \chi_k(p)$, we have:

\[ S_k(N) = \sum_{p\le N} \chi_k(p) = \sum_{r'\in (\ZZ/m\ZZ)^\times} \chi_k(r') \cdot \pi(N, m, r') \]

After computing these sums for all values of $k$, we are able to extract $\pi(N, m, r)$ for all values of $r$, by

\begin{equation}\label{eq:count-residue-classes-lin-comp} \pi(N, m, r) = \frac{1}{\varphi(m)}\sum_{k} \chi_k(r)^{-1} S_k(N) \end{equation}

This gives an $\tilde{O}(\sqrt N)$-time algorithm for computing $\pi(N, m, r)$ for any constant $m, r$. We now analyze effect of $m$ on the complexity. Consider the error term given by \Cref{eq:error-term-generalized} for each $\chi_k$:

\begin{equation}
\text{error term}_k = \sum_{\substack{d_1d_2>N \\ \kbar (d_1)+\khat (d_2)\le \kbar(N)}} \chi_k(d_1d_2)\cdot \mu_{\le \sqrt N}(d_2)
\end{equation}

Hence, the combined error term obtained by computing the approximation of the right-hand side of \Cref{eq:count-residue-classes-lin-comp} is:

\[ \text{combined error} =  \frac{1}{\varphi(m)}\sum_{\substack{d_1d_2>N \\ \kbar (d_1)+\khat (d_2)\le \kbar(N)}} \mu_{\le \sqrt N}(d_2)\cdot \sum_{k} \chi_k(r)^{-1} \chi_k(d_1d_2)  \]

The inner sum is 0 for any $d_1d_2$ that is not congruent to $r$ modulo $m$, hence:

\begin{equation} \text{combined error} = \sum_{\substack{n>N \\ n\equiv_m r}}\sum_{\substack{d_1d_2 = n \\ \kbar (d_1)+\khat (d_2)\le \kbar(N)}} \mu_{\le \sqrt N}(d_2)
\end{equation}

In other words, most errors cancel out, and we only need to iterate over the factorization of numbers $n\equiv r \pmod{m}$. It follows that the error term can be computed in $\tilde{O}\parens{\sqrt{N} + \frac{N\Delta}{m}}$ time, since we sieve only $1/m$ of the critical interval with primes up to $\sqrt N$.

On the other hand, computing the approximations of all the $S_k$'s takes $\tilde{O}\parens{\frac{m}{\Delta}}$, since we perform the $\tilde{O}\parens{\frac{1}{\Delta}}$-time FFT-based approximation algorithm $m$ times. The total time is thus $\tilde{O}\parens{\sqrt{N} + \frac{N\Delta}{m} + \frac{m}{\Delta}}$, which is optimized at $\Delta=\frac{m}{\sqrt{N}}$, with this complexity analysis valid as long as $m = \tilde{O}(\sqrt N)$. Hence, we get:

\begin{corollary}\label{cor:count-residue-classes}
Given $N$ and two co-prime integers $r < m$, with $m=\tilde{O}(\sqrt{N})$, we can compute $\pi(N, m, r)$ in $\tilde{O}\parens{\sqrt N}$ time. Note that the time is independent of $m$.
\end{corollary}

If we wish to compute the values of $\pi(N, m, r)$ for all values of $r$, given $N$ and $m$, we can now compute all error terms for the individual target values $\pi(N, m, r)$ in total time $\tilde{O}\parens{\sqrt{N} + N\Delta}$, by performing a single sieve to factorize all numbers in the critical interval, and then performing divisor-iteration for each residue class $r$ modulo $m$ separately.

The approximation time is still $\tilde{O}\parens{\frac{m}{\Delta}}$, and we also required $O(m^2)$ time\footnote{Solving this specific system of equations can in fact be done asymptotically faster using an appropriate fast DFT, but for simplicity we skip the details.} to extract all individual values of $\pi(N, m, r)$ given the values of $S_k(r)$. Hence the total time is 
\[\tilde{O}\parens{\sqrt{N} + N\Delta + \frac{m}{\Delta} + m^2} \]

By setting $\Delta=\sqrt\frac{m}{N}$, we obtain:

\begin{corollary}\label{cor:count-all-residue-classes}
Given $N$ and a positive integer $m=\tilde{O}(N^{1/3})$, we can compute $\pi(N, m, r)$ for all values of $r$ in $\tilde{O}\parens{\sqrt {m N}}$ time.
\end{corollary}

\subsection{Mertens function}\label{sec:mertens}
In this section we refer to the Mertens function $M(N)=\sum_{n=1}^N \mu(n)$.

Recent work by Helfgott and Thompson \cite{elementary-mertens} improved upon the previous best time complexity of $\tilde{O}\parens{N^{2/3}}$ for computing the Mertens function using an elementary method, obtaining a time complexity of $O\parens{N^{3/5}(\log N)^{3/5 + \epsilon}}$. We are able to further improve upon this, and compute the Mertens function in $O\parens{\sqrt N \log N \sqrt{\log \log N}}$ time.

One should note that an analytic method was already known to achieve $O\parens{N^{1/2+\epsilon}}$ time complexity \cite{lagarias_analytic}, but this has never been implemented, at least to the best of our knowledge.

\subsubsection{A naive extension of our \(\pi(N)\) method}
Here we briefly discuss an approach that is similar to the one taken so far for computing $\pi(N)$. This approach is presented for completeness, since the Mertens function allows for a simplified version of our method, which is both simpler and more efficient. This method is discussed in \Cref{sec:mertens_good}.

We start with the following identity:

\begin{lemma}\label{lem:mertens-from-smooth-mu}
The function:
\[ 2\mu_{\le \sqrt N} - \parens{\mu_{\le \sqrt N}\Conv \mu_{\le \sqrt N} \Conv \one} \]
equals the Möbius function $\mu$ for inputs not larger than $N$.
\end{lemma}
\begin{proof}
We are interested in the first $N$ values of

\[
\mu = \mu_{\le \sqrt N} \Conv \parens{\Conv_{\sqrt N < p \le N}\mu_p} = \mu_{\le \sqrt N} \Conv \parens{\Conv_{\sqrt N < p < N}(\delta_1 - \delta_p)}
\]

Any convolution $\delta_p\Conv \delta_{p'}$ for $p,p'>\sqrt N$ does not contribute to the result, since it only affects entries at inputs above $N$. Hence we are interested in the first entries of $\mu_{\le \sqrt N} \Conv \parens{\delta_1 - \sum_{\sqrt N < p \le N}\delta_p }$. But $\sum_{\sqrt N < p \le N}\delta_p$ is the same (in the first $N$ values) as $\mu_{\le \sqrt N}\Conv \one  - \delta_1$, hence this is finally equivalent to $\mu_{\le \sqrt N} \Conv \parens{\delta_1 - (\mu_{\le \sqrt N}\Conv \one - \delta_1))}$.
\end{proof}

We can efficiently compute the sum over these values using the techniques developed in this paper. More specifically, using a variant of \Cref{thm:error-term}.

\subsubsection{An improved algorithm}\label{sec:mertens_good}
We present here a variation of our method that is simpler and more efficient. Starting with \Cref{lem:mertens-from-smooth-mu}, one can arrive at the variant of Vaughan’s identity used also in \cite{elementary-mertens}:

\begin{equation}\label{lem:mergens-from-nonsmooth-mu}
M(N) = 2M(\sqrt N) - \sum_{n=1}^N \sum_{\substack{m_1m_2n_1=n \\ m_1,m_2\le \sqrt N}}\mu(m_1)\mu(m_2)
\end{equation}

We note that we can use $\mu$ instead of $\mu_{\le \sqrt N}$ in \Cref{lem:mergens-from-nonsmooth-mu}, since we only need it at indices up to $\sqrt N$.

We can compute \Cref{lem:mergens-from-nonsmooth-mu} in two phases. First, we sieve over $[1, \sqrt N]$ in $O(\sqrt N \log \log N)$ time, to compute $\mu(n)$ for each $n\le \sqrt N$. We are then able to trivially compute $2 M(\sqrt N)$. Next, we observe that the summation in \Cref{lem:mergens-from-nonsmooth-mu} is the sum of $\mu_{\text{trunc}}\Conv \mu_{\text{trunc}}\Conv \one$ up to and including $N$, where $\mu_{\text{trunc}}$ is $\mu$ truncated at $\sqrt N$, that is, $\mu_{\text{trunc}}(n)=\mu(n)$ for $n\le \sqrt N$ and 0 otherwise.

Here we do not need to consider prime factorizations, but directly work with the precomputed $\mu_\text{trunc}$. We use our exponential segmentation to obtain an array $\bar{\mu}_\text{trunc}$ of size $O\parens{\frac{\log N}{\Delta}}$. We are then able to use fast convolutions to compute

\begin{equation}
\bar{\mu}_\text{trunc} \Conv \bar{\mu}_\text{trunc} \Conv \bone
\end{equation}

in $O\parens{\frac{\log^2 N}{\Delta}}$ time (this time complexity follows from $1/\Delta$ ending up being $N^{O(1)}$, hence $O\parens{\log\parens{\frac{\log N}{\Delta}}} = O(\log N)$).

As with $\pi(N)$, we require an error correction phase. Analogous to \Cref{thm:error-term}, we have:

\begin{align*}
\begin{split}
\sum_{k=0}^{\kbar (N)} (\bone\Conv \bar{\mu}_\text{trunc}\Conv \bar{\mu}_\text{trunc})[k] - \sum_{n=1}^N (\one\Conv \mu_\text{trunc}\Conv \mu_\text{trunc})(n)\\ = \sum_{\substack{d_1d_2d_3>N \\ \kbar(d_1)+\kbar (d_2)+\kbar (d_3)\le \kbar(N)}} \one(d_1)\cdot \mu_\text{trunc}(d_2)\cdot \mu_\text{trunc}(d_3)
\end{split}
\end{align*}

This is slightly more complex than \Cref{thm:error-term} since we need to iterate over triplets $(d_1,d_2,d_3)$, but is also simpler by not having any $\khat$, but rather only $\kbar$'s. This means significantly fewer rounding errors exist, making the critical interval much smaller.

Indeed, since $\kbar(d_1)+\kbar (d_2)+\kbar (d_3)\ge \kbar(d_1d_2d_3) - 2$, we only need to iterate over divisor triplets of $n$ with $\kbar(N)\le \kbar(n)\le \kbar(N)+2$, so the critical interval has size $O(\Delta N)$. We can sieve to factorize all numbers in this interval in time $O(\sqrt N  + \Delta N \log\log N)$. Now, for each $n$ in the critical interval, we iterate over all triplets $(d_1,d_2,d_3)$ such that $d_1d_2d_3=n$, and accumulate the corresponding error-terms. The total work done in the error correction phase can be bounded using the following lemma.

\begin{lemma}\label{lem:bound-triple-divisors}
The number of triplets $(a,b,c)\in\NN^3$ such that $N < abc < N + O(\Delta N)$ is bounded by $O\parens{\Delta N \log^2 N}$.
\end{lemma}

\begin{proof}
We apply \Cref{lem:sum-small-interval-bound} to the function $f=\tau_3$, where $\tau_3(n)$ is the number of triplets $(a,b,c)$ such that $n=abc$. This is easily seen to be a multiplicative function. Also, $\tau_3\parens{p^\ell} = \binom{\ell+2}{2}$ satisfies the condition in the theorem. Hence, since $1/\Delta$ will be $N^{O(1)}$, we conclude from \Cref{lem:sum-small-interval-bound} that

\begin{align*}
\sum_{n=N}^{N+O(\Delta N)} \tau_3(n) &= O\parens{\frac{\Delta N}{\log N} \exp\parens{\sum_{p\le 2N} \frac{3}{p}}} \\ &= O\parens{\frac{\Delta N}{\log N} \exp\parens{3 \ln \ln (2N) + O(1)}}
\\ &= O\parens{\Delta N \log^2 N}
\end{align*}
\end{proof}

We are now ready to prove:

\begin{theorem}\label{thm:mertens-log2}
The Mertens function $M(N)$ can be computed in $O(\sqrt N \log^2 N)$ time.
\end{theorem}
\begin{proof}
As mentioned, we compute $\mu(n)$ up to $n=\sqrt N$ using a sieve. This requires $O(\sqrt N \log \log N)$ time. Then these values are used to compute $2M(\sqrt N)$ and the array $\bar{\mu}_{\text{trunc}}$ of size $O\parens{\frac{\log N}{\Delta}}$. We use fast convolutions to compute $\bar{\mu}_{\text{trunc}}\Conv \bar{\mu}_{\text{trunc}}\Conv \bone$ in $O\parens{\frac{\log^2 N}{\Delta}}$ time, then iterate over all divisor triplets in the critical interval in $O(\Delta N \log^2 N)$ time, given by \Cref{lem:bound-triple-divisors}. In total, the time is bounded by

\begin{equation}
O\parens{\sqrt N \log\log N + \frac{\log^2 N}{\Delta} + \Delta N \log^2 N}
\end{equation}

This is minimized when $\Delta=\Theta\parens{\frac{1}{\sqrt{N}}}$, which recovers a running time of $O(\sqrt N \log^2 N)$.
\end{proof}

\subsubsection{Further improvement using subset convolution}

We now present a faster and more sophisticated error correction phase.

\begin{lemma}\label{lem:mertens-error-correction-omega2tau}
Error correction phase for the Mertens function can be done in $O\parens{\tau(n)\omega(n)^2}$ time for each $n$ in the critical interval, where $\tau(n)$ is the number of divisors of $n$.
\end{lemma}
\begin{proof}
For each $n$ in the critical interval, we iterate over its prime factorization (computed in the sieving phase). Suppose $n=\prod_{i=1}^{\omega(n)} p_i^{\alpha_i}$ is its prime factorization, with $\alpha_i > 1$ for the first $r$ values of $i$, for some $r\ge 0$, and $\alpha_i=1$ otherwise. Recall we only need to iterate over triplets $(d_1,d_2,d_3)$ with $d_1d_2d_3=n$ and $d_2$, $d_3$ are square-free, and accumulate $\mu(d_2)\mu(d_3)$ whenever $\kbar(d_1)+\kbar(d_2)+\kbar(d_3)\le \kbar(N)$. We bypass the problem of checking whether $\kbar(d_1)+\kbar(d_2)+\kbar(d_3)\le \kbar(N)$, by precomputing a polynomial mapping $D$ such that for any $d\mid n$ we have $D[d](x)=x^{\kbar(d)}\mu(d)$. Then, the correction term will be the sum of coefficients of all $x^t$ for $t\le \kbar(N)$ in the polynomial $P_n$ defined by
\begin{equation}\label{eq:mertens-ec-pn}
P_n=\sum_{d_1d_2d_3=n} x^{\kbar(d_1)}\cdot D[d_2]\cdot D[d_3]
\end{equation}

However, we observe that, given $d_1d_2d_3=n$, we must have $\kbar(d_1)+\kbar(d_2)+\kbar(d_3) \in [\kbar(n)-2, \kbar(n)]$, and hence $P(x)/x^{\kbar(n)-2}$ is a polynomial of degree at most 2. Therefore, by evaluating it in any 3 points, we will be able to recover its coefficients, and sum the relevant ones into the error term. For efficiency, all calculations will take place in a finite field, one that is large enough to recover the result (for example, working modulo the same prime number used for the NTT on the initial $\bar{\mu}_{\text{trunc}}$ and $\bone$ arrays). Hence, from now on we will freely refer to $D[d]$ and $P_n$ as scalars, and not polynomials.

If $n$ is square-free (that is, $r=0$), there is a one-to-one map between divisors of $n$ and subsets of the primes dividing $n$. Let $\mathcal{P}$ be the set of primes dividing $n$. For a subset $S\subseteq \mathcal{P}$ we write $D[S]$ for the corresponding $D[d]$. Then, we begin by computing the \emph{subset convolution} $D_2=D\Conv D$, defined by $(D\Conv D)(S)=\sum_{T\subseteq S} D[T]\cdot D[S\setminus T]$. Using the Fast Subset Convolution algorithm from \cite{fast-subset-convolution}, the values of $D_2[S]$ for all subsets $S\subseteq \mathcal{P}$ can be computed together in time $|\mathcal{P}|^2 2^{|\mathcal{P}|} = \omega(n)^2 2^{\omega(n)} = \omega(n)^2\tau(n)$. Then, we can compute $P_n$ by:

\[
P_n = \sum_{S\subseteq \mathcal{P}} E[\mathcal{P}\setminus S] D_2[S] ,
\]
where $E[S]$ is defined as $x^{\kbar(d)}$ for the corresponding divisor $d$.

If $n$ is not square-free, the situation is slightly more complicated, but we can still reduce it to the square-free case. Suppose $r=1$, so there is exactly one prime $p$ dividing $n$ with multiplicity greater than one. Let $\alpha$ be this multiplicity. We can compute $P_n$ in \Cref{eq:mertens-ec-pn} by observing that $d_2$ and $d_3$ still need to be square-free to contribute. We can split into two cases: $p$ divides both $d_2$ and $d_3$, or at most 1 of them. The first case contributes 
\begin{equation}\label{eq:mertens-case-p-divides-d2d3}
\sum_{d_1d_2d_3=n/p^{\alpha}} x^{\kbar(d_1p^{\alpha-2})}\cdot D[d_2 p]\cdot D[d_3 p]
\end{equation}
which is again a 3-fold subset convolution, since $n/p^{\alpha}$ is square-free. Moreover, this convolution uses only $\omega(n) - 1$ primes. The second case contributes exactly
\begin{equation} \sum_{d_1d_2d_3=n/p^{\alpha-1}} x^{\kbar(d_1p^{\alpha-1})}\cdot D[d_2]\cdot D[d_3]
\end{equation}
which again can be computed via Fast Subset Convolution, since $n/p^{\alpha-1}$ is square-free.

Generalizing, for any $r>0$ primes dividing $n$ with multiplicity greater than $1$, we can reduce the computation to $2^r$ subset convolutions, by conditioning, for each prime, whether it divides both $d_2$ and $d_3$, or at most one of them. We now analyze the time required for all these cases.

For each $t\in[0,r]$, there are $\binom{r}{t}$ cases where a $t$ of the primes divides both $d_2$ and $d_3$. Each such case requires Fast Subset Convolution on a set of primes of size $\omega(n)-t$ (recall that by using \Cref{eq:mertens-case-p-divides-d2d3} we reduce the remaining number of primes to consider, where the other $r-t$ conditions do not reduce the number of primes), which therefore take $O\parens{\omega(n)^2 2^{\omega(n)-t}}$ time in total. The total work is, therefore, in the order of magnitude of:

\begin{align*}
\begin{split}
\sum_{t=0}^r & \binom{r}{t}\omega(n)^2 
 2^{\omega(n)-t} = \omega(n)^2 2^{\omega(n)}\sum_{t=0}^r \binom{r}{t} 
 2^{-t} \\ &= \omega(n)^2 2^{\omega(n)}\parens{1+\frac{1}{2}}^r = \omega(n)^2 \parens{\prod_{i=1}^r 3}\cdot \parens{\prod_{i=r+1}^{
\omega(n)} 2} \\ & \le \omega(n)^2 \prod_{i=1}^{\omega(n)} (\alpha_i+1) = \omega(n)^2 \tau(n)
\end{split}
\end{align*}
The last transition used the fact that $\alpha_i = 1$ for $i > r$ and $\alpha_i \ge 2$ otherwise.
\end{proof}

It remains to bound the sum $\sum_{n=N+1}^{N+S} \omega(n)^2 \tau(n)$. We now prove a tight upper bound on this sum.

\begin{lemma}\label{lem:mertens-bound-omega2tau}
Given, $\sqrt N < S<N$, the following holds
\[ \sum_{n=N+1}^{N+S}\omega(n)^2 \tau(n) = O\parens{S \log N (\log \log N)^2} \]
\end{lemma}
\begin{proof}
We first observe that it is enough to prove that

\begin{equation}\label{eq:mertens-sum-omega-pbound}
\sum_{n=N+1}^{N+S} \parens{\sum_{\substack{p\mid n \\ p < N^{1/5}}} 1}^2 \tau(n) = O\parens{S \log N (\log \log N)^2}
\end{equation}

The reason is that $\frac{1}{6}\omega(n) \le \sum_{\substack{p\mid n \\ p < N^{1/5}}} 1$ whenever there is a prime $p<N^{1/5}$ dividing $n$ (since there are at most 5 primes $\ge N^{1/5}$ dividing $n$), and any number with all prime factors above $N^{1/5}$ contributes at most a constant to the sum.

We now omit the subscript $p<N^{1/5}$ for ease of notation. We expand the left-hand side of \Cref{eq:mertens-sum-omega-pbound} as:

\begin{align*} \sum_{n=N+1}^{N+S} \parens{\sum_{p\mid n} 1}^2 \tau(n) &= \sum_{n=N+1}^{N+S} \parens{\sum_{p\mid n}\sum_{q\mid n} 1} \tau(n) = \sum_{p, q} \sum_{\substack{n\in[N+1, N+S] \\ p,q\mid n}} \tau(n)
\end{align*}

We further note now that we are free to discard terms with $p=q$, since they only account for at most a constant fraction of the result for each $n$, unless $n$ is prime (in which cast it only contributes a constant to the sum). Hence, using $\tau(npq)\le 4\tau(n)$ for any two primes $p,q$, it is enough to bound

\[ \sum_{p \neq q} \sum_{\substack{n\in(N, N+S] \\ p,q\mid n}} \tau(n) = \sum_{p \neq q} \sum_{n = \ceil*{(N+1)/pq}}^{\floor*{(N+S)/pq}} \tau(npq) \le 4\sum_{p \neq q} \sum_{n = \ceil*{(N+1)/pq}}^{\floor*{(N+S)/pq}} \tau(n)  \]

That is, for each two different primes $p,q\le N^{1/5}$, we are summing $\tau(n)$ in an interval of size $\frac{S}{pq} > \frac{N^{1/2}}{N^{2/5}} = N^{1/10}$. It follows that we can apply \Cref{lem:sum-small-interval-bound} to replace each $\sum_{n = \ceil*{(N+1)/pq}}^{\floor*{(N+S)/pq}} \tau(n)$ with $O\parens{\frac{S/pq}{\log N}\exp\parens{\sum_{p_2\le 2N} \frac{2}{p_2}}} = O\parens{\frac{S}{pq}\log N}$.

Finally, summing over all $p\neq q$, this is bounded by

\[ O\parens{S \log N \parens{\sum_{p<N^{1/5}} \frac{1}{p}}\parens{\sum_{q<N^{1/5}} \frac{1}{q}}} =O\parens{S \log N (\log \log N)^2} \]

as promised.
\end{proof}

We now have an improved version of \Cref{thm:mertens-log2}:

\begin{theorem}\label{thm:mertens-log15}
The Mertens function $M(N)$ can be computed in time

\[ O\parens{\sqrt N (\log N)^{3/2} \log \log N} . \]
\end{theorem}
\begin{proof}
Similar to the proof of \Cref{thm:mertens-log2}, and using \Cref{lem:mertens-error-correction-omega2tau} and \Cref{lem:mertens-bound-omega2tau} to bound the time spent on error correction, using $S=O(\Delta N)$, the total time is bounded by

\begin{equation*}
O\parens{\sqrt N \log\log N + \frac{\log^2 N}{\Delta} + \Delta N \log N (\log \log N)^2}
\end{equation*}

This is minimized when $\Delta=\Theta\parens{\frac{\sqrt{\log N}}{\sqrt{N} \log \log N}}$, obtaining a running time of $O\parens{\sqrt N (\log N)^{3/2} \log \log N}$.
\end{proof}

\subsubsection{Further improvement using tables}

The same look-up table method used in \Cref{sec:div-table} can be used here, mapping the multi-set of the approximated fractional prime factorization of each $n$ to the corresponding error term. Using a look-up table will further reduce the asymptotic running time needed for computing the Mertens function, making the Fast Subset Convolution technique unnecessary. This introduces extra randomness into the algorithm to enjoy the guarantees, and might also be less efficient in practice due to the size of the table for small values of $N$. We only sketch the details as they are very similar to those already discussed in \Cref{sec:div-table}.

Here, the condition for a triplet $(d_1,d_2,d_3)$ such that $d_1d_2d_3=n \in (N, N+S]$ to be corrected is that $d_2,d_3$ are square-free with $\pmax \le \sqrt N$, and that \begin{equation}\label{eq:mertens-ec-cond}
    \kbar(d_1)+\kbar(d_2)+\kbar(d_3)\le \kbar(N)
\end{equation}

If these conditions are met, we correct the result by $(-1)^{\omega(d_2)+\omega(d_3)}$. Using similar manipulations that led to \Cref{eq:table-div-criterion}, we rewrite \Cref{eq:mertens-ec-cond} as
\begin{equation}\label{eq:mertens-ec-cond-fractional}
\left\{\frac{\log_2 n}{\Delta}\right\} - \left\{\frac{\log_2 d_2}{\Delta}\right\}  - \left\{\frac{\log_2 d_3}{\Delta}\right\} < \kbar(N) - \kbar(n) + 1
\end{equation}

Here, as in \Cref{sec:lut}, it is enough to know $\left\{\frac{\log_2 n}{\Delta}\right\}$ and the factorization of $d_2$ and $d_3$ with enough precision, then with high probability we know how to account for the divisor triplet $(d_1,d_2,d_3)$ based on only the approximation.
We omit the details, and state the result:

\begin{lemma}\label{lem:mertens-table-ec} Using look-up tables, the error correction phase in the Mertens calculation can be done in $O\parens{S \log \log N}$ time.
\end{lemma}

As a corollary, we have the following result:

\begin{theorem}\label{thm:mertens-logn}
The Mertens function $M(N)$ can be computed in time
\[ O\parens{\sqrt N \log N \sqrt{\log \log N}} . \]
\end{theorem}

\subsection{Counting square-free numbers}\label{sec:sqfree}
We describe an $\tilde{O}\parens{N^{1/3}}$-time algorithm for computing the number of square-free numbers up to $N$, improving on the current state-of-the-art algorithm in $\tilde{O}\parens{N^{2/5}}$ time described in \cite{counting-sqfree}. The algorithm presented here combines our techniques with the ideas of \cite{counting-sqfree}.

As in \cite{counting-sqfree}, we start with an inclusion-exclusion on prime squares, giving the following expression for the number of square-free numbers $\le N$:

\begin{equation}\label{eq:sqfree-basic}
\operatorname{sqfree}(N) = \sum_{d=1}^{\floor*{\sqrt N}} \mu(d)\floor*{\frac{N}{d^2}}
\end{equation}

We derive a different expression based on \Cref{eq:sqfree-basic} as follows:

\begin{align*}
\begin{split}
&\sum_{d=1}^{\floor*{\sqrt N}} \mu(d)\floor*{\frac{N}{d^2}} = \sum_{d \le \sqrt N} \sum_{t\le N/d^2} \mu(d) \\
& = \sum_{t\le N} \sum_{d \le \sqrt{N/t}}  \mu(d) = \sum_{t\le N} M\parens{\sqrt{\frac{N}{t}}}
\end{split}
\end{align*}

That is:
\begin{equation}\label{eq:sqfree-mertens}
\operatorname{sqfree}(N) = \sum_{t\le N} M\parens{\sqrt{\frac{N}{t}}}
\end{equation}

Now, as in \cite{counting-sqfree}, we will first sieve to compute all $\mu(d)$ values for $d<D$, in $O(D\log\log D)$ time. This takes care of all values of $t\ge \frac{N}{D^2}$ in \Cref{eq:sqfree-mertens}. To compute all other values of the Mertens function, we apply the methods developed in \Cref{sec:mertens} for computing the Mertens function, with a slight change: we set a precision variable $\Delta$, perform a single FFT to obtain the array $\bar{\mu}_\text{trunc} \Conv \bar{\mu}_\text{trunc} \Conv \bone$, and do a separate error correction for each of the $\sqrt{\frac{N}{t}}$ thresholds. Crucially, we perform a single FFT that is useful for computing all these values.

While for $t>1$ this amounts to applying a formula different from that given in \Cref{lem:mergens-from-nonsmooth-mu} (as we are using $\mu(m_1),\mu(m_2)$ for $m_1,m_2$ greater than the square root of the input in these cases), this can still be done, because in \Cref{lem:mergens-from-nonsmooth-mu} we can replace $\sqrt N$ in both occurrences in the right-hand side with any threshold $T$ larger than that, as long as we use the same threshold $T$ on both parts of the formula:
\[
M(N) = 2M(T) - \sum_{n=1}^N \sum_{\substack{m_1m_2n_1=n \\ m_1,m_2\le T}}\mu(m_1)\mu(m_2)
\]

This can be justified by observing that we never add a term $\mu(m_1)\cdot \mu(m_2)$ with \emph{both} $m_1,m_2>\sqrt N$, since otherwise $m_1m_2n_1$ would be too large to be included. Then, any specific $\mu(m)$ for $m$ between $\sqrt N$ and $T$ is counted in the sum $2\sum_{n} \mu(n)\floor*{\frac{N}{mn}}=2$ times (as this sum counts by inclusion-exclusion the integers up to $N/m$ divisible by no prime), exactly canceling its contribution to the $2M(T)$ term.

Looking at \Cref{eq:sqfree-mertens}, the maximal input required to the Merten function is $\sqrt N$. The length of the arrays is $O\parens{\frac{\log N}{\Delta}}$, the cost of FFT is $O\parens{\frac{\log^2 N}{\Delta}}$, and then for each value of $\sqrt{N/t}$ we need an error correction phase taking $\tilde{O}\parens{\Delta \sqrt{\frac{N}{t}}}$ time.

The total cost is then

\[
\tilde{O}\parens{D + \frac{1}{\Delta} + \Delta \sum_{t=1}^{\floor*{N/D^2}} \sqrt{\frac{N}{t}}}
\]
which, using $\sum_{k=1}^m \frac{1}{\sqrt k} = O(\sqrt m)$, can be simplified into
\[
\tilde{O}\parens{D + \frac{1}{\Delta} + \Delta\frac{N}{D}}
\]
which is minimized for $D=\Delta^{-1}=\tilde{\Theta}\parens{N^{1/3}}$, finally giving $\tilde{O}\parens{N^{1/3}}$ time.

Thus, we proved:

\begin{theorem}\label{thm:sqfree-sum}
$\operatorname{sqfree}(N)$ can be computed in $\tilde{O}{\parens{N^{1/3}}}$ time.
\end{theorem}

\subsection{Totient summatory function}
As a final demonstration of the generality of our techniques, we show how to efficiently compute the totient summatory function
\[ \Phi(N)=\sum_{n=1}^N \varphi(n) \]
where $\varphi$ is Euler's totient function. To the best of our knowledge, the best algorithm known for computing $\Phi(N)$ runs in $\tilde{O}\parens{N^{2/3}}$ time, with the possible exception of an $O\parens{N^{1/2+\epsilon}}$ analytical approach.

\begin{theorem}\label{thm:euler-sum}
$\Phi(N)$ can be computed in $\tilde{O}\parens{\sqrt N}$ time.
\end{theorem}

\begin{proof}[Proof sketch]
Our algorithm uses the following identity:

\begin{equation}
\Phi(N)=\sum_{n=1}^N \mu(n)\cdot \frac{1}{2} \floor*{\frac{N}{n}}\parens{1 + \floor*{\frac{N}{n}}} ,
\end{equation}
which follows by observing that $\Phi(N)$ counts pairs of co-prime numbers up to $N$, and applying inclusion-exclusion on the prime factorization. We replace $\frac{1}{2} \floor*{\frac{N}{n}}\parens{1 + \floor*{\frac{N}{n}}}$ by $\sum_{1\le k_1 \le k_2 \le N/n} 1$, change the order of summation, and finally obtain

\begin{equation}\label{eq:totient-sum-by-mobius}
\Phi(N)=\sum_{k=1}^N k\cdot M\parens{\frac{N}{k}}
\end{equation}

We now proceed in the same way we did for computing the square-free numbers up to $N$ (\Cref{sec:sqfree}), using a single FFT for computing the Mertens function on all values up to $N$, without a sieve for computing smaller values. The time required for the whole algorithm is then

\[ \tilde{O}\parens{\frac{1}{\Delta} + \Delta \sum_{k=1}^{\floor*{N/D}} \frac{N}{k}} = \tilde{O}\parens{\frac{1}{\Delta} + \Delta N} \]
which gives $\tilde{O}(\sqrt N)$ when $\Delta=\frac{1}{\sqrt N}$.
\end{proof}

\section{Implementation considerations}
\subsection{Avoiding rounding and floating-point errors}
The only floating-point calculations necessary in our algorithms are in computing $\kbar(n)$, where we only need to ensure that $\kbar(n)$ is non-decreasing, and roughly increases logarithmically in order to enjoy the guarantees given in the theorems. This ensures no floating-point errors can occur.

The only other part requiring floating-point calculations is in the computations involving the look-up tables, which are necessary only in order to improve about a $O(\sqrt {\log N})$ factor in the time bound, and is more sensitive to the precision in order to enjoy the time guarantees (but is easily guaranteed correctness).

\subsection{Implementation speed-ups}

\subsubsection{Smaller modulus in NTT}\label{sec:smaller-modulus}
We note that the NTT (that is, finite-field FFT) modulus must be a prime number larger than the padded arrays we convolve, that is, it must be at least $2\frac{\log_2 N}{\Delta} = \tilde{O}(\sqrt N)$. Since the result of the algorithm, the number of primes, is $\tilde{O}(N)$, one option is to carry the entire computation modulo a prime which is $\tilde{O}(N)$.

An alternative is to compute the result modulo a product of two primes, each large enough for the NTT. This requirement already guarantees that their product is larger than the result of the algorithm. Working with two primes is beneficial because we can run the entire algorithm separately on each prime, working with half the word size, and then combine the two results at the end using the Chinese Remainder Theorem.

\subsubsection{Assuming the Riemann Hypothesis}
The Riemann Hypothesis provides a bound $\pi(N)=\operatorname{li}(N) + O(\sqrt N \log N)$. Thus, by assuming the Riemann Hypothesis, it is enough to compute the result modulo a number that is $O(\sqrt N \log N)$, then find the unique representative in that interval. This means that unlike the proposal of the previous subsection, we can work with a single prime modulus of size $\tilde{O}(\sqrt N)$. This saves a factor of two in the running time.

\subsubsection{Convolving separately with small primes}\label{sec:conv-small-primes}
In the subsequent improvements, it will be helpful to be able to not deal with small primes. In order to not disrupt the main flow of ideas, we describe here how we are able to deal separately with small primes in a different way. Indeed, our end goal is to compute $\hmu$. We can compute it first on primes $p\ge \log_2 N$, discarding all small primes. Then, we can convolve with $\delta_{0} - \delta_{\kbar (p)}$ over primes $p < \log_2 N$ (recall \Cref{eq:hmu-convolution}). This last convolution can be done in-place in linear time for each prime, and so this requires an additional $O\parens{\frac{\sqrt N}{\Delta } \log N}$ time to compute $\hmu$ given the incomplete computation. This is dominated by other parts of our algorithm, and saves a constant factor from the FFT work.

\subsubsection{Discarding small primes}\label{sec:discard-small-primes}
We can remove all occurrences of small primes from the initial $\bone$ array. That is, we produce an adjusted $\bone$ array, where each segment counts the number of numbers divisible by only primes above some threshold $p_\text{min}$.

This alteration has to be met with removing these small primes from the convolution computing $\hmu$. It can be seen that then, the rest of the logic still holds, and we are still computing $\pi(N)-\pi(\sqrt N) + 1$ as in \Cref{lem:pi-relation}.

The advantage, now, is that fewer divisors need to be checked for error correction, because there are fewer "rounding errors". In other words, $\khat(d)$ is now closer, on average, to $\kbar(d)$.

Note that this does not affect the maximal $r$ used for the $C_r$'s, because the primes were already partitioned into smaller ranges to counter the effect variation in prime sizes.

To apply this change, we need to be able to efficiently compute the adjusted $\bone$ array. We observe that the number of integers not divisible by any prime $p\le p_\text{min}$ is constant across intervals of size exactly $\prod_{p\le p_\text{min}} p$, where their count is exactly $\prod_{p\le p_\text{min}} (p-1)$. It follows that for each segment, we can compute the amount of numbers in it by first finding the number of intervals of size $\prod_{p\le p_\text{min}} p$ that fit in it, and accounting for the remainder using a precomputed table of size $\prod_{p\le p_\text{min}} p$. This then requires $O\parens{\prod_{p\le p_\text{min}} p}$ time to perform.

Since $\prod_{p\le p_\text{min}} p = e^{(1+o(1))p_\text{min}}$, we can use $p_\text{min} \approx \frac{1}{2}\ln N$, requiring $O(\sqrt N)$ time, that is dominated by other runtimes in the algorithm.

\subsection{Integration with the combinatorial method}\label{sec:combinatorial-integration}
Asymptotically, our algorithm has a better time complexity than the combinatorial approach. Nevertheless, the latter is much better in practice for smaller values of $N$.

Here we would like to suggest that our approach can be combined with the combinatorial approach. One way to do so is to split the primes between the two approaches: the combinatorial approach can efficiently count the numbers up to each threshold of the form $N/n$ that are coprime to all primes up to a certain threshold $B$. These counters then can be used to replace the $\bone$ array with a version that only requires the removal of numbers that are divisible by primes in the range $[B, \sqrt{N}]$. Using our method we can do so by convolving this array and a modified Möbius function $\hat{\mu}_{[B, \sqrt{N}]}$ that only includes numbers divisible only by these primes. This is an improvement of the idea discussed in \Cref{sec:discard-small-primes}.

By using both methods we can utilize the advantages of both approaches. The combinatorial method handles small primes very efficiently, while our method enjoys significantly easier error correction by considering only numbers without small factors and noting that the number of factors of such numbers is bounded.

The above idea provides a continuous trade-off between the combinatorial method and our method. We believe that there will be a transition zone, for numbers $N$ where our method does not out-performs the combinatorial approach, but a combination of the two approaches does so.

We remark that we do not know of a way to combine the combinatorial approach with the space-optimized version of our algorithm.

\subsection{Parallelization}
Most of our algorithm is easily parallelizable:
\begin{itemize}
    \item Subsets of primes can be handled in parallel.
    \item For each such subset, once the initial FFT is completed, the sequence $c_r$ can be computed in parallel for different Fourier coefficients.
    \item Convolving the partial Möbius function for different prime subsets can be done in parallel using a binary tree.
    \item Error correction for each segment (value of $\kbar(n)$) can be done in parallel.
    \item Sieve can be done in parallel using a segmented sieve.
    \item Accessing the table of \Cref{sec:div-table} and handling table misses for different $n$'s is completely parallelizable.
\end{itemize}

It is noted that the ideas presented for reducing the space complexity make it possible to easily \emph{distribute} the computations of the whole algorithm. In the FFT phase, for example, it is possible to compute the $Z_{\ell_0}$ arrays (as described in \Cref{sec:space-improvement-details}) for different values of $\ell_0$ independently.

\section{Relations to the analytic method}
In this section we discuss similarities between our approach and the analytic method.

We start with presenting a variation on our algorithm that avoids segmentation and rounding. We believe this version to be almost identical to the analytical method. We do not establish a concrete algorithm in this section, since this algorithm would be very similar to the one presented in papers discussing the analytic method \cite{lagarias_analytic}.

\subsection{Blurring instead of segmentation}
We start by noting that the basic version of our algorithm (\Cref{sec:basic-algorithm}) removes from $\one$ numbers that are divisible by primes $\le \sqrt N$ using convolution with a smooth Möbius function. This is not analogous to the analytic method. In \Cref{lem:ln-identity} we present an identity that for $N^{1/t} < 2$ allows to filter all non-prime values using only convolutions of $\one$ with itself. We discuss here why this variant is analogous to the analytic method.

Rather than approximating Dirichlet convolution using array convolutions, we can view Dirichlet convolution as a convolution of (generalized) functions over $\RR$. In other words, given a function $f$ over $\NN$, we represents it using the following generalized function over $\RR$:
  \[ \tilde{f}(x) = \sum_{n \in \NN} f(n) \delta(x - \ln n) \]
where $\delta(x)$ is the Dirac delta function. We observe that indeed the convolution of the representations of two functions is the representation of their Dirichlet convolution:
\[\begin{split}
  (\tilde{f} \Conv \tilde{g})(x) &= 
  \int \tilde{f}(t) \tilde{g}(x-t) dt = \\
  &= \parens{\sum_{n_0 \in \NN} f(n_0) \delta(x - \ln n_0)}\Conv \parens{\sum_{n_1 \in \NN} g(n_1) \delta(x - \ln n_1)} = \\
  &= \sum_{n_0 \in \NN} \sum_{n_1 \in \NN} f(n_0) g(n_1) \parens{\delta(x - \ln n_0) \Conv \delta(x - \ln n_1)} = \\
  &= \sum_{n_0 \in \NN} \sum_{n_1 \in \NN} f(n_0) g(n_1) \delta(x - \ln n_0 - \ln n_1) = \\
  &= \sum_{n \in \NN} \parens{\sum_{d|n} f(d) g(n/d)} \delta(x - \ln n) \\
  &= \sum_{n \in \NN} (f \Conv g)(n) \delta(x - \ln n)
\end{split}\]
We used the identity $\delta(x - x_0) \Conv \delta(x - x_1) = \delta(x - x_0 - x_1)$.

As in the algorithm presented in this paper, the computation begins with a representation of all numbers $\le N$ which we denote
 \[ \tilde{\one}(x) = \sum_{n=1}^N \delta(x-\ln n) \]

The obvious problem with this approach is that there is no fast way to conduct convolutions of such generalized functions over $\RR$. We would like to use the Fourier transform:
  \[ F(\xi) = \mathcal{F}\{\tilde{f}\}(\xi) = \int_{-\infty}^\infty \tilde{f}(x) e ^ {-2 \pi i \xi x} dx \]

However, $F(\xi)$ takes values over all $\xi \in \RR$, so it cannot be fully computed. Instead, we wish to evaluate the Fourier transform only at a finite set of values. We restrict $F$ to multiples of a fundamental frequency $\xi = k \xi_0$. That is, we replace $F(\xi)$ with:
 \[ \sum_{k \in \ZZ} F(k \xi_0) \delta(\xi - k\xi_0) \]
This is the same as multiplying $F(\xi)$ by $\sum_k \delta(\xi - k\xi_0)$, which is the same as convolving the original function with $\sum_k \delta(x - k/\xi_0)$. In other words, we compute a cyclic folding of $f$ over a period of $1/\xi_0$. As in our algorithm, we would like to have this period large enough to avoid wrap around. A reasonable choice would be $1/\xi_0 = O(\log ^2 N)$, allowing for the representation of numbers up to $N$ and their $\log_2 N$ powers, needed for manipulating such functions.

We are left with the task of computing and manipulating the Fourier transform at points $\xi = k \xi_0$. Recall that our algorithm starts with the function $\tilde{\one}(x) = \sum_{n=1}^N \delta(x-\ln n)$. Its Fourier transform at the desired points is:
  \[ \mathcal{F}\{\tilde{\one}\}(k \xi_0) = 
  \sum_{n=1}^N e ^ {-2 \pi i k \xi_0 \ln n} = 
  \sum_{n=1}^N n ^ {-2 \pi i k \xi_0} = 
  \zeta_N \parens{2 \pi i k \xi_0}\]

Here $\zeta_N(s)=\sum_{n=1}^N n^{-s}$ is the Riemann zeta function truncated at $n=N$. In other words, we need to evaluate the (truncated) zeta function at selected points along the imaginary axis. The analytic method uses an analogous procedure that evaluates the truncated zeta function simultaneously at many points, and then applies the Riemann–Siegel formula to approximate the Riemann zeta function using the truncated value.

Even though we restricted $F$ to a countable set of points $\xi = k \xi_0$, this set is still infinite. We must add another restriction: $|\xi| < \xi_{\text{max}}$. This has the same effect as multiplying $F(\xi)$ by a rectangle window with this span, which translates to blurring the original function with the kernel $\operatorname{sinc}(\xi_{\text{max}} x) = \frac{\sin(\xi_{\text{max}}x)}{\xi_{\text{max}}x}$.

Blurring the original function means that the computation would not be able to recover the exact number of primes up to $N$. Instead, the result would essentially count primes larger than $N$, though the contribution of such primes diminishes as $p$ increases. Similarly, primes slightly smaller than $N$ will not contribute exactly $1$ to the sum. The rate of decay is determined by the rate of decay of the blurring kernel, that is, the integral of the tail of $\operatorname{sinc}(\xi_{\text{max}} x)$. This decays very slowly, approximately as $\frac{1}{\xi_{\text{max}} x}$. This was a problem especially in the early attempts towards the analytic method.

The breakthrough came when \cite{lagarias_analytic} solved the issue by replacing the sinc kernel with a rapidly decaying one. In other words, instead of summing up the Fourier coefficients up to a threshold $\xi_{\text{max}}$, they proposed to sum them up with weights that corresponds to an improved kernel. \cite{galway-analytic} proposed to use a gaussian kernel for this purpose. This kernel decays very rapidly both in the original space and in Fourier space. Using this kernel, we can choose $\xi_{\text{max}} = \tilde{O}(\sqrt{N})$ and get a blurring with width $\Delta x = \tilde{O}(1 / \sqrt{N})$. Overall, this method requires the computation of $2\xi_{\text{max}} / \xi_0 = \tilde{O}(\sqrt{N})$ Fourier coefficients and in addition applying error correction for contributions of interval of size $\Delta n = N \cdot \Delta x = \tilde{O}(\sqrt{N})$ around $N$. Unlike our method, the analytic method avoids segmentation so it only has contributions from prime numbers. This somewhat simplifies error correction, and even allows the use of a sieve faster by a factor of $\ln \ln N$, like the Atkin sieve.

\subsection{Advantages of our approach}
We believe that our approach offers several advantages over the analytic method.

First, we believe that our approach is simpler to understand, to analyze and to implement. This is why we consider our approach "elementary", unlike the analytic method. One manifestation is that we were able to obtain an accurate complexity bound for our method. To the best of our knowledge, no such complexity analysis is known for the analytic method, the complexity is only known up to $N ^ {\epsilon}$.

The analytic method relies heavily on complex analysis, while our approach avoids it entirely by using Fourier transform as a black-box for fast convolution.

A related difference is that the analytic method carries the computation in complex numbers, which are prone to numerical errors. This, together with the use of complex analysis, means the analytic method requires a sophisticated error analysis in order to justify the correctness of its result (even though in practice the fact that it is very close to an integer may serve as an indication for correctness).

In our approach the computation is done entirely in integers. We use finite-field FFT. In addition to avoiding numerical errors, this enables working with numbers of smaller precision, as discussed in \Cref{sec:smaller-modulus}. Our approach does have one step that uses real numbers: we need to evaluate $\kbar(n)$ and its inverse. This amounts to computing $\log$ or $\exp$ in some precision. We note that if the computation is done correctly, numerical errors here do not change the correctness of the result, as the error correction phase can use the same computation and fix false contributions that were caused by numerical errors.

Lastly, the combinatorial method is currently more efficient than the analytic method, despite being asymptotically faster. We believe that our approach may be combined with the combinatorial method to obtain an algorithm that defeats both methods for intermediate values of $N$. This idea is briefly presented in \Cref{sec:combinatorial-integration}. We note that unlike our method, the analytic method can not be combined in such a way due to the fact that it does not have a combinatorial interpretation.

\subsection{Advantages of the analytic method}
One significant advantage of the analytic method over our method is its space complexity. It achieves its optimal time complexity $O(N^{1/2 + \epsilon})$ using space as small as $O(N^{1/4 + \epsilon})$. Our approach requires $\tilde{O}(\sqrt N)$ space in order to achieve $\tilde{O}(\sqrt N)$ time. Using less space than that would degrade the time complexity (\Cref{thm:pi-space-time-tradeoff}). Even though this degradation is not huge, it may be enough for the analytic method to win on memory-bounded systems.

We hope that the space complexity of our method may be reduced by improving the ideas presented in \Cref{sec:space-improvement-details}.

\section*{Acknowledgement}
We thank Ohad Klein for helpful ideas on reducing the space complexity of our algorithm (\Cref{sec:space-improvement}). We also thank Noam Kimmel, Gal Porat, Amir Sarid and Roee Sinai for helpful comments on earlier drafts of this paper.

\newpage
\bibliography{main}

\newpage
\appendix

\section{Sketched details of space improvements}\label{sec:space-improvement-details}

We denote by $M$ the allowed space complexity for the algorithm. The forthcoming analysis holds for any value of $N^{\epsilon} < M < \sqrt N$.

The main goal of this section is to give a proof sketch of \Cref{lem:bone-space-time}.

\subsection{Smaller FFTs for the \texorpdfstring{$\bone$}{one} array}
As in \Cref{sec:small-fft-for-mobius}, we use an analog of \Cref{eq:low-mem-fft-formula-mobius}:

\[
\widetilde{\bone}[\ell_0 + L_0\ell_1] =
\sum_{m_1=0}^{L_1-1}  \parens{ \zeta_L ^ {L_0} } ^ {\ell_1 m_1}\sum_{\substack{n: \kbar(n)\le \kbar(N) \\ \kbar(n) \equiv m_1 (\text{mod } L_1)}} \zeta_L ^ {\ell_0 \kbar(n)}
\]

It follows that:

\begin{equation}\label{eq:low-mem-fft-formula-bone}
\widetilde{\bone}[\ell_0 + L_0\ell_1] =
\sum_{m_1=0}^{L_1-1}  \parens{ \zeta_L ^ {L_0} } ^ {\ell_1 m_1}\sum_{\substack{k\le \kbar(N) \\ k \equiv m_1 (\text{mod } L_1)}} \bone[k] \zeta_L ^ {\ell_0 k},
\end{equation}

where $\bone[k]$ is the number of $n$'s such that $\kbar(n)=k$:

\[
\bone[k] = \ceil*{2^{\Delta (k+1)}} - \ceil*{2^{\Delta k}}
\]

As before, we fix the value of $\ell_0$, and compute the $L_1$ values of $\widetilde{\bone}[\ell]$ for $\ell\equiv \ell_0 \pmod{L_0}$. We have shown that the values of $\widetilde{\bone}[\ell_0+L_0\ell_1]$ are the Fourier transform of order $L_1$ of the array $Z$ defined by:

\begin{equation}\label{eq:bone-fourier-z} Z_{\ell_0}[m_1] = \sum_{\substack{k\le \kbar(N) \\ k \equiv m_1 (\text{mod } L_1)}} \bone[k] \zeta_L ^ {\ell_0 k}
\end{equation}

As in the array that represents the set of prime numbers, we are able to efficiently compute the Fourier transform of $\bone$ using memory proportional to its number of nonzero elements. We have shown that the array prime numbers can be reduced to primes up to $N^{1/t}$. This is not the case for $\bone$, which has $\tilde{O}(\sqrt{N})$ nonzero elements.

For reasons that will be made clear in the subsequent analysis, we will choose

\begin{equation}\label{eq:l1-size}
    L_1= \floor*{\frac{1}{\sqrt \Delta \ln 2}} w
\end{equation}

for some $w\in\NN$. This will also coincide (up to logarithmic factors) with the space limit $M$, hence we have $
L_1=\tilde{\Theta}(M)$ and $w=\tilde{\Theta}(\sqrt \Delta M)$.

Now, since $L_0 L_1$ is the total array size $\kbar(N)=\tilde{\Theta}(1/\Delta)$, and since $L_1=\tilde{\Theta}(M)$, we have:

\begin{equation}\label{eq:l0-size}
    L_0= \tilde{\Theta}\parens{\frac{1}{ M \Delta}}
\end{equation}

\subsection{Perturbed \texorpdfstring{$\bone$}{one}} \label{sec:perturbed-bone}
To reduce the memory consumption of the $\bone$ we use the fact that even though it is not very sparse, its entropy is very low. More specifically, we use the fact that locally $\bone$ is approximately periodic. We next discuss what "approximately" means in this context.

Instead of computing the FFT of $\bone$, we compute the FFT of a perturbed version denoted by $\bone^{\text{pert}}(n)$. We use the fact that slightly changing the values of $\kbar(n)$ introduces a new kind of error, but this error may be corrected using the same method used to correct segmentation errors.

In other words, we replace $\kbar(n)$ with a perturbed version $\kbar^{\text{pert}}(n)$. We require $\kbar^{\text{pert}}(n)$ to be non-decreasing and efficiently computable. In addition, we assume
\[|\kbar^{\text{pert}}(n) - \kbar(n)| = \log^{O(1)}(N). \]

Then, the usual error correction procedure can be used to correct the new kind of errors using $\kbar^{\text{pert}}(n)$ instead of $\kbar(n)$. This gives rise to corresponding definitions of all other related quantities, such as $Z_{\ell_0}^{\text{pert}}[m_1]$.

\subsubsection{Sparse part of \texorpdfstring{$\bone$}{one}}
For $k$ up to some threshold $K_{\text{sparse}}$, we compute the contribution to $Z_{\ell_0}$ directly by \Cref{eq:bone-fourier-z}. This requires $O(1)$ time per nonzero cell $\bone[k]$ for computing $Z_{\ell_0}$ for a given $\ell_0$. As there are $O\parens{2^{\Delta K_{\text{sparse}}}}$ nonzero entries of $\bone[k]$ up to $K_{\text{sparse}}$ and $L_0$ values of $\ell_0$, handling this part of the array requires overall time of

\begin{lemma}[informal]\label{lem:sparse-bone-time}
The time required for handling the sparse part of $\bone$ is
\[
O\parens{2^{\Delta K_{\text{sparse}}}L_0} = \tilde{O}\parens{\frac{2^{\Delta K_{\text{sparse}}}}{\Delta M}}
\]
\end{lemma}
where we have used \Cref{eq:l0-size}.

\subsubsection{Dense part of \texorpdfstring{$\bone$}{one}}
Recall that $\bone[k] = \ceil*{2^{\Delta (k+1)}} - \ceil*{2^{\Delta k}}$. Using the fact that $\frac{d}{dk} 2^{\Delta k} = 2^{\Delta k} \Delta \ln 2$, we expect $\bone[k]$ to be positive for $k > \frac{1}{\Delta} \log_2 \frac{1}{\Delta \ln 2}$. We call this the dense part of $\bone$ and denote its starting index by $K_\text{dense}$. Note that $K_\text{dense} = \tilde{\Theta}(\frac{1}{\Delta})$ and that the number $n_\text{dense} = \frac{1}{\Delta \ln 2}$ is mapped to this cell: $\kbar(n_\text{dense}) = K_\text{dense}$.

Instead of computing the contribution of the dense part of $\bone$ to its Fourier transform, we omit this contribution. That is, we actually replace $\bone$ with $\bone^{\text{trunc}}$ that is identical to $\bone$ for indices below $K_{\text{dense}}$, and 0 otherwise. We also denote by $\bone^{\text{dense}}=\bone - \bone^{\text{trunc}}$ the dense part of $\bone$.

This introduces a new kind of error to the result. In order to correct this error, note that we compute convolutions of the form $\bone \Conv \hat{\mu}_{\le N^{1/t}}$ and sum up to $N$. To fix the calculation made with only $\bone^{\text{trunc}}$, we simply need to compute $\bone^{\text{dense}} \Conv \hat{\mu}_{\le N^{1/t}}$, whose entries can be expanded into:

\begin{align*}
(\bone^{\text{dense}} \Conv \hat{\mu}_{\le N^{1/t}})[k] &= \sum_{i=0}^k \bone^{\text{dense}}[i] \cdot \hat{\mu}_{\le N^{1/t}}[k-i] \\ &= \sum_{i=K_{\text{dense}}}^k \bone^{\text{dense}}[i] \cdot \hat{\mu}_{\le N^{1/t}}[k-i]
\end{align*}

This uses the value of $\hat{\mu}_{\le N^{1/t}}$ only at indices up to $\kbar(N)-K_{\text{dense}}$, which corresponds to numbers at most about $N/ n_{\text{dense}} = O(\Delta N)$. Hence, we can compute all values $\mu_{\le N^{1/t}}(n)$ up to this threshold (using, e.g., a segmented sieve as described in \Cref{sec:sieve-small-segments}), and multiply each by the appropriate sum of indices of $\bone^{\text{dense}}$.

However, since we use \Cref{lem:ln-identity} to reduce the space complexity, we actually need to compute a more complicated convolution. For example, we might have a term of the form $\bone\Conv\bone\Conv\hmt\Conv\hmt$. In order to correct values omitted from $\bone$, we replace $\bone$ by $\bone^{\text{trunc}} + \bone^{\text{dense}}$ and expand. The resulting terms that use the $\bone^{\text{dense}}$ array are:
\[ \bone^{\text{dense}} \Conv \parens{(2\cdot \bone^{\text{trunc}} + \bone^{\text{dense}}) \Conv\hmt\Conv\hmt} \]
or more generally $\bone^{\text{dense}} \Conv g$ for some array $g$. We then only need to compute $g$ at indices $\kbar(n)$ that correspond to numbers $n \le N / n_\text{dense}$, and we can proceed to compute the resulting correction array (and hence its sum) in $\tilde{O}(\Delta N)$ time.

Concluding, the error introduced by omitting the dense part of $\bone$ can be corrected in $\tilde{O}(N / n_\text{dense}) = \tilde{O}(\Delta N)$ time. We do not bother with further improvements, as this is exactly the time complexity required for correcting segmentation errors (see \Cref{sec:error-correction}).

\subsubsection{Middle part of \texorpdfstring{$\bone$}{one}}
For $K_{\text{sparse}} < k < K_\text{dense}$ we expect each integer number $n$ to be mapped to a unique value of $\kbar(n) = \floor*{\frac{\log_2 n}{\Delta}}$. For any $n_0$ we know that:
\[
  \frac{\log_2 n}{\Delta} = \frac{\log_2 n_0}{\Delta} + \frac{n - n_0}{\Delta n_0 \ln 2} + O\parens{\frac{(n - n_0)^2}{\Delta n_0^2}} 
\]
for any $n \ge n_0$. Recall that we may perturb $\kbar(n)$ by up to $\log^{O(1)}(N)$, so we may choose to replace $\kbar(n)$ with the linear approximation:
\[
  \kbar^{\text{pert}}(n) = \floor*{\frac{\log_2 n_0}{\Delta} + \frac{n - n_0}{\Delta n_0 \ln 2}}
\]
as long as $n_0 \le n < n_0 (1 + \sqrt{\Delta})$, because in this range the error (before taking the floor) is bounded by $\log^{O(1)}(N)$.

We note that $n = n_0 (1 + \sqrt{\Delta})$ is mapped by $\kbar(n)$ to

\[\floor*{\frac{\log_2 n}{\Delta}} \approx \frac{\log_2 n_0}{\Delta} + \frac{\sqrt{\Delta}}{\Delta \ln 2} = \frac{\log_2 n_0}{\Delta} + \frac{1}{\sqrt{\Delta} \ln 2} \]

The above derivation shows that we can pick any nonzero element in $\bone$ and replace the next $\frac{1}{\sqrt{\Delta} \ln 2}$ indices of this array with values that are derived by a first-order approximation of $\kbar$ at the starting nonzero element $k_0 = \kbar(n_0)$. We denote the interval size by $m = \floor*{\frac{1}{\sqrt{\Delta} \ln 2}}$. The nonzero indices in that interval form an arithmetic progression (or, more precisely, a rounding of an arithmetic progression to integer values) with step $\frac{1}{\Delta n_0 \ln 2}$.

Summarizing, the sequences of indices in which $\bone$ is nonzero in the $\bone$ vector (and hence $\bone$ at these indices is 1) are slightly modified so that they can be partitioned into arithmetic progressions spanning non-intersecting intervals of the array. The $t$-th arithmetic sequence of nonzero indices of $\bone$ is characterized by:

\begin{itemize}
\item $n_{0, t} = \Theta\parens{2^{K_{\text{sparse}} \Delta + t\sqrt \Delta / \ln 2}}$.
\item Starts at index $s_t = \kbar(n_{0, t}) = K_{\text{sparse}} + \frac{t}{\sqrt \Delta \ln 2} + O(1)$.
\item Has step size (difference) of $d_t=\Theta\parens{\frac{1}{\Delta n_{0, t}}}=\Theta\parens{\frac{2^{- K_{\text{sparse}} \Delta}}{\Delta}2^{-t\sqrt \Delta / \ln 2}}$.
\item Has $c_t=\Theta\parens{\frac{1/\sqrt \Delta}{d_t}} = \Theta\parens{ \sqrt \Delta 2^{K_{\text{sparse}} \Delta}2^{t\sqrt \Delta / \ln 2}}$ elements.
\item Recalling \Cref{eq:l1-size}, we will choose $L_1$ and $\Delta$ such that $L_1$ is a multiple of $m = \floor*{\frac{1}{\sqrt{\Delta} \ln 2}}$ (recall $(c_t - 1) d_t + (s_t \text{ mod }m) < m$, but $c_td_t + (s_t \text{ mod }m) \ge m$. This choice will reduce cumbersome technicalities in the following calculations).
\item The number of sequences $t$ is at most
\[t_{\text{max}}\le \frac{k_{\text{max}}}{1/(\sqrt \Delta \ln 2)} - \frac{K_{\text{sparse}}}{1/(\sqrt{\Delta} \ln 2)}\le \frac{\ln 2}{\sqrt{\Delta}} \log_2 \frac{\ln 2}{\Delta}  - \sqrt{\Delta} \ln 2 K_{\text{sparse}}.\]
\end{itemize}

We will now describe how, for each such sequence, we are able to efficiently compute its contribution to $Z_{\ell_0}$ for a fixed $\ell_0 < L_0$. Recall \Cref{eq:bone-fourier-z}:

\[Z_{\ell_0}[m_1] = \sum_{\substack{k\le \kbar(N) \\ k \equiv m_1 (\text{mod } L_1)}} \bone[k] \zeta_L ^ {\ell_0 k}
\]

The contribution of the $t$-th sequence (parameterized as $\floor*{s_t + i d_t}$ for $i=0,\ldots,c_t-1$), denoted by $\text{contrib}_{\ell_0, t}[m_1]$, is:

\[
\text{contrib}_{\ell_0, t}[m_1] = 
\sum_{\substack{0\le i\le c_t-1 \\ \floor*{s_t + i d_t} \equiv m_1 (\text{mod } L_1)}} \zeta_L ^ {\ell_0 \floor*{s_t + i d_t}}
\]

Letting $s_t' = s_t\text{ mod }L_1$, since $(c_t-1) d_t + (s_t'\text{ mod }m) < m$ by assumption, and since $L_1$ is a multiple of $m$, we can rewrite this as 

\begin{equation}\label{eq:contrib-formula}
\text{contrib}_{\ell_0, t}[m_1] = 
\sum_{\substack{0\le i\le c_t-1 \\ \floor*{i d_t + s_t'}=m_1}} \zeta_L ^ {\ell_0 \floor*{s_t + i d_t}}
\end{equation}

For a given $\ell_0$, we can trivially account for the contribution from the $t$-th sequence in $O(c_t)$ time. This will be the preferred method for sequences with small $c_t$, which will be included as indices below $K_{\text{sparse}}$. For large values of $c_t$ we will require a different method that we next describe.

Before proceeding, we make an important observation. We see that in \Cref{eq:contrib-formula} (and the discussion preceding it), each sequence only updates one of the $\frac{L_1}{m}$ contiguous blocks of size $m$ of entries in $Z_{\ell_0}$. Hence, we will separately encode each such block of $m$ entries in $L_0$ by a separate array of size $m$, where each sequence now updates exactly one of the arrays. We will consider contributions to each of these blocks separately, where \Cref{eq:contrib-formula} still holds with $s_t'=s_t\text{ mod }m$ instead of mod $L_1$.

\subsection{Updating an arithmetic progression symbolically}
In order to efficiently aggregate the arithmetic progressions, we encode $Z_{\ell_0}[m_1]$ for a fixed $\ell_0$ and all values of $m_1$ by a polynomial $Z_{\ell_0}(x)$ defined by

\begin{equation}
Z_{\ell_0}(x) = \sum_{m_1=0}^{L_1-1} Z_{\ell_0}[m_1] x^{m_1}
\end{equation}

As observed in the discussion following \Cref{eq:contrib-formula}, contributions are made to one of the $L_1/m$ contiguous blocks of $m$ entries in $Z_{\ell_0}$, and it will be convenient for us to encode in a separate polynomial of degree $m-1$ the contributions to each such block.

Then, defining $\text{contrib}_{\ell_0, t}(x)$ analogously, we have:

\begin{equation}
\text{contrib}_{\ell_0, t}(x) = 
\sum_{i=0}^{c_t-1} \zeta_L ^ {\ell_0 \floor*{s_t + i d_t}} x^{\floor*{i d_t + s_t'}}
\end{equation}

If all the $d_t$'s were integers, we could perform the following manipulation. Working with formal power series modulo $x^{m}$, and since $c_td_t + s_t' \ge m$, we could write 

\[
\text{contrib}_{\ell_0, t}(x) \equiv \sum_{i=0}^{\infty} \zeta_L ^ {\ell_0 (s_t + i d_t)} x^{i d_t + s_t'} \equiv \frac{\zeta_L ^ {\ell_0 s_t}x^{s_t'}}{1-\zeta_L ^ {\ell_0 d_t} x^{d_t}} \pmod {x^{m}}
\]

However, this is not justified when $d_t$ is not an integer -- the treatment with formal power series assumes we are encoding values with integral powers of $x$. Instead, we approximate $d_t$ with a rational fraction.

It is well-known (by Dirichlet's approximation theorem) that it is possible to approximate any real number to within $1/Q^2$ using a denominator bounded by $O(Q)$. Hence, there are integers $u_t, v_t$ such that $\left|\frac{u_t}{v_t} - d_t\right| \le \frac{1}{c_t}$ and $v_t = O(\sqrt{c_t})$.
Replacing $d_t$ with this fraction ensures that the error incurred in any value of the arithmetic sequence $s_t+id_t$ for $0\le i \le c_t-1$ is bounded by 1. Replacing the sequence by $\floor*{s_t + i\frac{u_t}{v_t}}$ then changes $\kbar$ by only $O(1)$, a change we are allowed to impose (see \Cref{sec:perturbed-bone}). It also follows that $u_t \approx d_t v_t = O(d_t \sqrt c_t)$.

Working with power series in $x$ modulo $x^{m}$, we see we can write

\[
\text{contrib}_{\ell_0, t}(x) \equiv \sum_{i=0}^{\infty} \zeta_L ^ {\ell_0 \floor*{s_t + i u_t/v_t}} x^{\floor*{s_t' + i u_t/v_t}}  \pmod {x^{m}}
\]

Note now that the term at index $i+v_t$ is a multiple of the term at index $i$, with the ratio being $x^{u_t} \zeta_L^{\ell_0 u_t}$, hence this is a sum of $v_t$ geometric series with the same ratio. It follows we can write

\[
\text{contrib}_{\ell_0, t}(x) \equiv \frac{p_t (y)}{1-y^{u_t}} \pmod {y^{m}}
\]

where $y=x\zeta_L^{\ell_0}$ and $p_t$ is an appropriate polynomial with degree smaller than $u_t$ having $v_t$ nonzero coefficients.

We are now able to add all contributions from the sequences (that is, from all different values of $t$), in the following way: first, we add together the numerators $p_t (y)$ of all contributions with the same denominator, which can be done in linear time. Partial sums are always kept in the form $\frac{p(y)}{q(y)}$ for polynomials $p,q$ of degree less than $m$ and $\deg p < \deg q$. Adding two such expressions $\frac{p_1(y)}{q_1(y)} + \frac{p_2(y)}{q_2(y)}$ is reduced to computing $p_1q_2+q_2p_2$ and $q_1q_2$, both using FFT-based polynomial multiplication in $\tilde{O}(\deg p_1 + \deg p_2)$ time.

Finally, we arrive at an expression of the form $\frac{p(y)}{q(y)}$, from which we can finish by computing the first $m$ terms of the inverse power series of $q(y)$, the multiplying by $p(y)$, in $\tilde{O}(m)$ time.

Adding all such quotients $\frac{p_t(y)}{1-y^u}$ in a degree-balanced binary-tree structure (until degrees are close to $m$) then requires time proportional, up to logarithmic factors, to the sum of degrees in all polynomials used. Recalling we first merged all terms with the same denominator, for each of the $L_1/m$ blocks of $m$ entries, the required time is

\begin{align*}
\tilde{O}\parens{\sum_{u \le u_{\text{max}}} u} &= \tilde{O}\parens{u_{\text{max}}^2} = \tilde{O}\parens{\max_t (d_t \sqrt {c_t})^2} = \tilde{O}\parens{\frac{1}{\sqrt\Delta} \max_t d_t} \\ &= \tilde{O}\parens{\frac{1}{\sqrt\Delta} \frac{2^{- K_{\text{sparse}} \Delta}}{\Delta}}
\end{align*}

We then multiply this time by $\frac{L_1}{m}$ to account for the number of blocks, and by $L_0$ to account for the different values of $\ell_0$ for which this whole process is repeated, the total time for adding the polynomials is

\begin{lemma}[informal]\label{lem:rat-func-add-time}
The total time required for adding rational functions is given by
\[
\tilde{O}\parens{\frac{2^{- K_{\text{sparse}} \Delta}}{\Delta^2}}
\]
\end{lemma}

where we have used the fact that $L_0L_1/m = \tilde{O}
(1/\Delta) \cdot \tilde{O}(\sqrt \Delta) = \tilde{O}(1/\sqrt{\Delta})$. This is independent of $M$.

We have yet to account for the time required to sum-up all numerators of $\frac{p_t(y)}{1-y^{u_t}}$ with the same denominator. Each takes $O(v_t)$ time, which brings the total time per value of $\ell_0$ to 

\begin{align*}
    O\parens{\sum_{t} v_t} &= \tilde{O}\parens{\sum_{t} \sqrt{c_t}} = \tilde{O}\parens{ \sum_t \Delta^{1/4} 2^{K_{\text{sparse}} \Delta / 2}2^{t\sqrt \Delta / 2\ln 2}} \\ &= \tilde{O}\parens{\Delta^{1/4} 2^{K_{\text{sparse}} \Delta / 2}\cdot \frac{1}{\sqrt \Delta} \cdot 2^{t_{\text{max}}\sqrt \Delta / 2\ln 2}} \\ &= \tilde{O}\parens{ \frac{1}{\Delta^{3/4}}}
\end{align*}

Since this computation is repeated for every value of $\ell_0$, this is multiplied by $L_0=\tilde{O}\parens{\frac{1}{\Delta M}}$ (see \Cref{eq:l0-size}), giving a total numerator-summation time of

\begin{lemma}[informal]\label{lem:numerator-v-summation}
The total numerator-summation time is given by
\[
\tilde{O}\parens{\frac{1}{\Delta^{7/4} M}}
\]
\end{lemma}

This is independent of the value of $K_{\text{sparse}}$.

Finally, we mention that the total FFT time on the $Z_{\ell_0}$ arrays is $\tilde{O}(L_0 L_1) = \tilde{O}(1/\Delta)$, which will be dominated by other running times in all cases.

\subsection{Complexity analysis}

Summing up, using \Cref{lem:rat-func-add-time} and \Cref{lem:numerator-v-summation} as the total time required for the FFT computations of the middle part of the $\bone$ array, \Cref{lem:sparse-bone-time} for the sparse part of $\bone$, and \Cref{lem:low-mem-mobius} for the $\widetilde{\hat{\mu}}_{\le N^{1/t}}$ array, the total running time with memory $M$ is given by
\[
\tilde{O}\parens{\frac{1}{\Delta^{7/4} M} + \frac{2^{- K_{\text{sparse}} \Delta}}{\Delta^2} + \frac{2^{\Delta K_{\text{sparse}}}}{\Delta M}}
\]

Optimizing over $K_\text{sparse}$ gives $2^{K_\text{sparse} \Delta} = \sqrt{M/\Delta}$, and the time is
\[
\tilde{O}\parens{\frac{1}{\Delta^{7/4} M} + \frac{1}{\Delta^{3/2} M^{1/2}}}
\]

When optimizing over $\Delta$, we need to account for the fact that $L_1 \ge m = \tilde{\Theta}(1/{\sqrt \Delta})$ requiring the space $M$ to be $\tilde{\Theta}(1/{\sqrt \Delta})$ or larger. In this domain, the resulting complexity is:
\[
\tilde{O}\parens{\frac{1}{\Delta^{3/2} M^{1/2}}}
\]

This informally proves \Cref{lem:bone-space-time}.

We mention that it is possible to obtain a different tradeoff for $M < 1/{\sqrt \Delta}$ down to $M=N^{\epsilon}$. We omit the discussion of this case here.

\end{document}